\documentclass[12pt,reqno]{article}

\usepackage[usenames]{color}
\usepackage{amssymb}
\usepackage{graphicx}
\usepackage{amscd}

\usepackage[colorlinks=true,
linkcolor=webgreen,
filecolor=webbrown,
citecolor=webgreen]{hyperref}

\definecolor{webgreen}{rgb}{0,.5,0}
\definecolor{webbrown}{rgb}{.6,0,0}

\usepackage{color}
\usepackage{fullpage}
\usepackage{float}

\usepackage{graphics,amsmath,amssymb}
\usepackage{amsthm}
\usepackage{amsfonts}
\usepackage{latexsym}
\usepackage{epsf}

\newcommand{\seqnum}[1]{\href{http://oeis.org/#1}{\underline{#1}}}

\usepackage{relsize,mathtools,cancel} 
\usepackage{qsymbols}
\usepackage{url} 

\usepackage[font=small,labelfont=bf,textfont=rm,format=hang]{caption,subcaption} 
\usepackage{rotating}

\newcommand{\tagonce}[0]{
     \addtocounter{equation}{1}
     \tag{\theequation}
} 
\newcommand{\tagtext}[1]{\tag{\footnotesize\underline{\emph{#1}}}}

\renewcommand{\labelenumi}{$\mathsmaller{\blacktriangleright}$ } 
\newcommand{\sublabel}[1]{\smallskip\noindent{\smaller{\underline{\textbf{#1}}.\ }} \bigskip \\ \noindent}

\newcommand{\quotetext}[1]{``#1''} 
\newcommand{\cf}[0]{cf.\ } 
\newcommand{\ie}[0]{i.e.,\ } 

\newcommand{\OEISII}[1]{{\texttt{#1}}} 
\newcommand{\seqmapsto}[2][]{
     \xrightarrow[\text{ \OEISII{#1} }]{\text{ \OEISII{#2} }}\quad
}  

\newcommand{\defequals}{\ensuremath{\vcentcolon=}} 
\newcommand{\defmapsto}{\ensuremath{\vcentcolon\mapsto}} 
\newcommand{\undersetbrace}[2]{\ensuremath{\underset{\mathlarger{#1}}{\mathsmaller{\underbrace{#2}}}}} 
\newcommand{\undersetline}[2]{\ensuremath{\underset{\mathlarger{#1}}{\mathsmaller{\underline{#2}}}}} 

\newcommand{\TheSummaryNBFileGoogleDriveLink}[0]{https://drive.google.com/file/d/0B6na6iIT7ICZRFltbTVVcmVpVk0/view?usp=drivesdk}  
\newcommand{\TheSummaryNBFile}{ 
     \href{\TheSummaryNBFileGoogleDriveLink}{\texttt{multifact-cfracs-summary.nb}}
}
\newcommand{\Mm}[0]{\emph{Mathematica}} 
\newcommand{\SigmaPkg}[0]{%
     \href{http://www.risc.jku.at/research/combinat/software/Sigma/index.php}{%
     \texttt{Sigma}}
}

\newcommand{\gkpSI}[2]{\ensuremath{\genfrac{\lbrack}{\rbrack}{0pt}{}{#1}{#2}}} 
\newcommand{\gkpSII}[2]{\ensuremath{\genfrac{\{}{\}}{0pt}{}{#1}{#2}}}

\newcommand{\FcfII}[3]{\ensuremath{\gkpSI{#2}{#3}_{#1}}} 
 
\newcommand{\FFactII}[2]{\ensuremath{#1^{\underline{#2}}}} 
 
\newcommand{\RFactII}[2]{\ensuremath{#1^{\overline{#2}}}} 
\newcommand{\Pochhammer}[2]{\ensuremath{\left(#1\right)_{#2}}} 
\newcommand{\Iverson}[1]{\ensuremath{\left[#1\right]_{\delta}}} 
\newcommand{\MultiFactorial}[2]{\ensuremath{#1!_{\left(#2\right)}}} 
\newcommand{\AlphaFactorial}[2]{\ensuremath{\left(#1\right)!_{\left(#2\right)}}} 
\newcommand{\pn}[3]{\ensuremath{p_{#1}\left(#2, #3\right)}}

\newcommand{\ConvGF}[4]{\ensuremath{\Conv_{#1}\left(#2, #3; #4\right)}} 
\newcommand{\ConvFP}[4]{\ensuremath{\FP_{#1}\left(#2, #3; #4\right)}}

\def\?{\hbox{!`}} 

\newcommand{\StartGroupingSubEquations}{\begin{subequations}} 
\newcommand{\EndGroupingSubEquations}{\end{subequations}} 

\DeclareMathOperator{\FP}{FP} 
\DeclareMathOperator{\FQ}{FQ} 
\DeclareMathOperator{\Conv}{Conv}

\DeclareMathOperator{\Log}{Log} 
\DeclareMathOperator{\RHS}{RHS} 

\DeclareMathOperator{\WT}{WT} 
\DeclareMathOperator{\CT}{CT} 
\DeclareMathOperator{\SPT}{SPT} 
\newcommand{\Wilson}{\mathsmaller{\WT}} 
\newcommand{\Clement}{\mathsmaller{\CT}} 
\newcommand{\SPTriple}{\mathsmaller{\SPT}} 

\DeclareMathOperator{\XT}{T} 
\DeclareMathOperator{\XK}{K} 
\DeclareMathOperator{\XP}{P} 
\newcommand{\xt}{x_{\mathsmaller{\XT}}} 
\newcommand{\xk}{x_{\mathsmaller{\XK}}} 
\newcommand{\xp}{x_{\mathsmaller{\XP}}} 

\DeclareMathOperator{\quot}{quot} 
\newcommand{\WilsonQuotient}[1]{W_{\mathsmaller{\quot}}\left(#1\right)} 

\DeclareMathOperator{\WilsonPrime}{Wilson} 
\DeclareMathOperator{\WolstPrime}{Wolstenholme} 
\DeclareMathOperator{\WieferichPrime}{Wieferich} 
\newcommand{\WilsonPrimeSet}{\mathbb{P}_{\mathsmaller{\WilsonPrime}}} 
\newcommand{\WolstPrimeSet}{\mathbb{P}_{\mathsmaller{\WolstPrime}}} 
\newcommand{\WieferichPrimeSet}{\mathbb{P}_{\mathsmaller{\WieferichPrime}}} 

\allowdisplaybreaks

\begin{document}

\theoremstyle{plain}
\newtheorem{theorem}{Theorem}
\newtheorem{lemma}[theorem]{Lemma}
\newtheorem{prop}[theorem]{Proposition}
\newtheorem{cor}[theorem]{Corollary}

\theoremstyle{definition}
\newtheorem{definition}[theorem]{Definition}
\newtheorem{example}[theorem]{Example}

\theoremstyle{remark}
\newtheorem{remark}[theorem]{Remark}

\begin{center}
\vskip 1cm{\LARGE\bf 
New Congruences and Finite Difference Equations for 
Generalized Factorial Functions} 
\vskip 1cm
\large
Maxie D. Schmidt \\ 
University of Washington \\ 
Department of Mathematics \\ 
Padelford Hall \\ 
Seattle, WA 98195 \\ 
USA \\ 
\href{mailto:maxieds@gmail.com}{\nolinkurl{maxieds@gmail.com}} 
\end{center}

\vskip .2 in
\begin{abstract} 
We use the rationality of the generalized $h^{th}$ convergent functions, 
$\Conv_h(\alpha, R; z)$, to the infinite J-fraction expansions 
enumerating the generalized factorial product sequences, 
$p_n(\alpha, R) = R(R+\alpha)\cdots(R+(n-1)\alpha)$, 
defined in the references to construct new congruences and $h$-order 
finite difference equations for generalized factorial functions modulo 
$h \alpha^t$ for any primes or odd integers $h \geq 2$ and 
integers $0 \leq t \leq h$. 
Special cases of the results we consider within the article include 
applications to new congruences and exact formulas for the 
$\alpha$-factorial functions, $n!_{(\alpha)}$. 
Applications of the new results we consider within the article include 
new finite sums for the $\alpha$-factorial functions, restatements of 
classical necessary and sufficient conditions of the primality of 
special integer subsequences and tuples, and new finite sums
for the single and double factorial functions modulo integers $h \geq 2$. 
\end{abstract} 

\vskip 0.2in 

\section{Notation and other conventions in the article} 
\label{Section_Notation_and_Convs} 

\subsection{Notation and special sequences} 

Most of the conventions in the article are consistent with the 
notation employed within the \emph{Concrete Mathematics} reference, and 
the conventions defined in the introduction to the first articles 
\cite{MULTIFACT-CFRACS,MULTIFACTJIS}. 
These conventions 
include the following particular notational variants: 
\begin{enumerate} 
     \renewcommand{\labelenumi}{$\mathsmaller{\blacktriangleright}$ } 
     \newcommand{\localitemlabel}[1]{\textbf{#1}.\ } 

\item \localitemlabel{Extraction of formal power series coefficients} 
The special notation for formal 
power series coefficient extraction, 
$[z^n] \left( \sum_{k} f_k z^k \right) \defmapsto f_n$; 

\item \localitemlabel{Iverson's convention} 
The more compact usage of Iverson's convention, 
$\Iverson{i = j} \equiv \delta_{i,j}$, in place of 
Kronecker's delta function where 
$\Iverson{n = k = 0} \equiv \delta_{n,0} \delta_{k,0}$; 

\item \localitemlabel{Bracket notation for the Stirling number triangles} 
We use the alternate bracket notation for the Stirling number triangles, 
$\gkpSI{n}{k} = (-1)^{n-k} s(n, k)$ and 
$\gkpSII{n}{k} = S(n, k)$; 

\item \localitemlabel{Harmonic number sequences} 
Use of the notation for the first-order harmonic numbers, $H_n$ or 
$H_n^{(1)}$, which defines the sequence 
\[
H_n \defequals 1+\frac{1}{2}+\frac{1}{3}+\cdots+\frac{1}{n}, 
\] 
and the notation for the partial sums for the more general cases of the 
$r$-order harmonic numbers, $H_n^{(r)}$, defined as 
\[ 
\mathsmaller{H_n^{(r)}} \defequals 1 + 2^{-r} + 3^{-r} + \cdots + n^{-r}, 
\]
when $r, n \geq 1$ are integer-valued and where $H_n^{(r)} \equiv 0$ for all 
$n \leq 0$; 

\item \localitemlabel{Rising and falling factorial functions} 
We use the convention of denoting the 
falling factorial function by $\FFactII{x}{n} = x! / (x-n)!$, the 
rising factorial function as $\RFactII{x}{n} = \Gamma(x+n) / \Gamma(x)$, 
or equivalently by the Pochhammer symbol, 
$\Pochhammer{x}{n} = x(x+1)(x+2) \cdots (x+n-1)$; 

\item \localitemlabel{Shorthand notation in integer congruences and modular arithmetic} 
Within the article the notation 
$g_1(n) \equiv g_2(n) \pmod{N_1, N_2, \ldots, N_k}$ is understood to 
mean that the congruence, $g_1(n) \equiv g_2(n) \pmod{N_j}$, holds 
modulo any of the bases, $N_j$, for $1 \leq j \leq k$. 

\end{enumerate} 
The standard set notation for 
$\mathbb{Z}$, $\mathbb{Q}$, and $\mathbb{R}$ 
denote the sets of integers, rational numbers, and real numbers, respectively, 
where the set of natural numbers, $\mathbb{N}$, is defined by 
$\mathbb{N} \defequals \{0, 1, 2, \ldots \} = \mathbb{Z}^{+} \bigcup \{0\}$. 
Other more standard notation for the special functions 
cited within the article is consistent with the definitions 
employed in the \emph{NIST Handbook of Mathematical Functions} (2010). 

\subsection{Mathematica summary notebook document and 
            computational reference information} 
\label{subSection_MmSummaryNBInfo} 

The article is prepared with a more extensive set of 
computational data and software routines 
released as open source software to accompany the examples and 
numerous other applications suggested as topics for future 
research and investigation within the article. 
It is highly encouraged, and expected, that the 
interested reader obtain a copy of the summary notebook reference and 
computational documentation prepared in this format to 
assist with computations in a multitude of special case examples cited as 
particular applications of the new results. 

The prepared summary notebook file, \TheSummaryNBFile, 
attached to the submission of this manuscript 
contains the working \Mm{} code to verify the formulas, 
propositions, and other identities cited within the article 
\cite{SUMMARYNBREF-STUB}. 
Given the length of this and the first article, 
the \Mm{} summary notebook included with this submission 
is intended to help the reader with verifying and 
modifying the examples presented as 
applications of the new results cited below. 
The summary notebook also contains numerical data 
corresponding to computations of multiple examples and 
congruences specifically referenced in several places by the 
applications given in the next sections of the article. 

\section{Introduction} 
\label{Section_Intro} 

\subsection{Motivation} 

In this article, we extend the results from the reference 
\cite{MULTIFACT-CFRACS} providing infinite J-fraction expansions for the 
typically divergent ordinary generating functions (OGFs) of generalized 
factorial product sequences of the form 
\begin{align} 
\label{eqn_GenFact_product_form} 
\pn{n}{\alpha}{R} & \defequals 
     R (R + \alpha) (R + 2\alpha) \times \cdots \times (R + (n-1)\alpha) 
     \Iverson{n \geq 1} + \Iverson{n = 0}, 
\end{align} 
when $R$ depends linearly on $n$. Notable special cases of 
\eqref{eqn_GenFact_product_form} that we are particularly interested in 
enumerating through the convergents to these J-fraction expansions include the 
multiple, or \emph{$\alpha$-factorial functions}, $\MultiFactorial{n}{\alpha}$, 
defined for $\alpha \in \mathbb{Z}^{+}$ as 
\begin{equation} 
\label{eqn_nAlpha_Multifact_variant_rdef} 
\MultiFactorial{n}{\alpha} = 
     \begin{cases} 
     n \cdot (n-\alpha)!_{(\alpha)}, & \text{ if $n > 0$; } \\ 
     1, & \text{ if $-\alpha < n \leq 0$; } \\ 
     0, & \text{ otherwise, } 
     \end{cases} 
\end{equation} 
and the generalized factorial functions of the form 
$p_n(\alpha, \beta n+\gamma)$ for $\alpha, \beta, \gamma \in \mathbb{Z}$, 
$\alpha \neq 0$, and $\beta, \gamma$ not both zero. 
The second class of special case products are related to the 
\emph{Gould polynomials}, 
$G_n(x; a, b) = \frac{x}{x-an} \cdot \FFactII{\left(\frac{x-an}{b}\right)}{n}$, 
through the following identity 
(\cite[\S 3.4.2]{MULTIFACTJIS},\cite[\S 4.1.4]{UC}): 
\begin{align} 
\label{eqn_pnAlphaBetanpGamma_GouldPolyExp_Ident-stmt_v1} 
\pn{n}{\alpha}{\beta n + \gamma} & = 
     \frac{(-\alpha)^{n+1}}{\gamma-\alpha-\beta} \times 
     G_{n+1}\left(\gamma-\alpha-\beta; -\beta, -\alpha\right). 
\end{align} 
The $\alpha$-factorial functions, $\AlphaFactorial{\alpha n-d}{\alpha}$ for 
$\alpha \in \mathbb{Z}^{+}$ and some $0 \leq d < \alpha$, form special cases of 
\eqref{eqn_pnAlphaBetanpGamma_GouldPolyExp_Ident-stmt_v1} where, equivalently, 
$(\alpha, \beta, \gamma) \equiv (-\alpha, \alpha, -d)$ and 
$(\alpha, \beta, \gamma) \equiv (\alpha, 0, \alpha-d)$ 
\cite[\S 6]{MULTIFACT-CFRACS}. 
The $\alpha$-factorial functions are expanded by the triangles of 
\emph{Stirling numbers of the first kind}, $\gkpSI{n}{k}$, and the 
\emph{$\alpha$-factorial coefficients}, $\FcfII{\alpha}{n}{k}$, respectively, 
in the following forms \cite{GKP,MULTIFACTJIS}: 
\begin{align*} 
\tagonce\label{eqn_AlphaNm1pd_AlphaFactFn_PolyCoeffSum_Exp_formula-eqns_v3} 
\MultiFactorial{n}{\alpha} & = 
     \sum_{m=0}^{n} 
     \gkpSI{\lceil n / \alpha \rceil}{m} 
     (-\alpha)^{\lceil \frac{n}{\alpha} \rceil - m} n^{m},\ 
     \forall n \geq 1, \alpha \in \mathbb{Z}^{+} \\ 
     & = 
     \sum_{m=0}^{n} 
     \FcfII{\alpha}{\lfloor \frac{n-1+\alpha}{\alpha} \rfloor + 1}{m+1} 
     (-1)^{\lfloor \frac{n-1+\alpha}{\alpha} \rfloor - m} 
     (n+1)^{m},\ 
     \forall n \geq 1, \alpha \in \mathbb{Z}^{+} \\ 
(\alpha n-d)!_{(\alpha)} 
     & = 
     (\alpha - d) \times 
     \sum_{m=1}^{n} \FcfII{\alpha}{n}{m} (-1)^{n-m} 
     (\alpha n + 1 - d)^{m-1} \\ 
     & = 
     \phantom{(\alpha - d) \times} 
     \sum_{m=0}^{n} \FcfII{\alpha}{n+1}{m+1} (-1)^{n-m} 
     (\alpha n + 1 - d)^{m},\ 
     \forall n \geq 1, \alpha \in \mathbb{Z}^{+}, 0 \leq d < \alpha. 
\end{align*} 
A careful treatment of the polynomial expansions of these generalized 
$\alpha$-factorial functions through the coefficient triangles in 
\eqref{eqn_AlphaNm1pd_AlphaFactFn_PolyCoeffSum_Exp_formula-eqns_v3} 
is given in the reference \cite{MULTIFACTJIS}. 

\subsection{Summary of the J-fraction results} 

For all $h \geq 2$, we can generate the generalized factorial product 
sequences, $p_n(\alpha, R)$, through the strictly rational 
generating functions provided by the $h^{th}$ convergent functions, 
denoted by $\ConvGF{h}{\alpha}{R}{z}$, to the infinite 
continued fraction series established by the reference \cite{MULTIFACT-CFRACS}. 
In particular, we have series expansions of these convergent functions 
given by 
\begin{align} 
\notag 
\ConvGF{h}{\alpha}{R}{z} & := 
     \cfrac{1}{1 - R \cdot z - 
     \cfrac{\alpha R \cdot z^2}{ 
            1 - (R+2\alpha) \cdot z -
     \cfrac{2\alpha (R + \alpha) \cdot z^2}{ 
            1 - (R + 4\alpha) \cdot z - 
     \cfrac{3\alpha (R + 2\alpha) \cdot z^2}{ 
     \cfrac{\cdots}{1 - (R + 2 (h-1) \alpha) \cdot z}}}}} \\ 
\label{eqn_ConvGF_notation_def} 
     & \phantom{:} = 
     \frac{\FP_h(\alpha, R; z)}{\FQ_h(\alpha, R; z)} \\ 
\notag 
     & \phantom{:} = 
     \sum_{n=0}^{h} p_n(\alpha, R) z^n + 
     \sum_{n>h}^{\infty} \left[p_n(\alpha, R) \pmod{h}\right] z^n,  
\end{align} 
where the convergent function numerator and denominator polynomial 
subsequences providing the characteristic expansions of 
\eqref{eqn_ConvGF_notation_def} are given in closed-form by 
\begin{align} 
\label{eqn_PFact_Qhz_product_ident} 
\FQ_h(\alpha, R; z) & = 
     \sum_{k=0}^{h} \binom{h}{k} (-1)^{k} \left(\prod_{j=0}^{k-1} 
     (R + (h-1-j)\alpha)\right) z^k \\ 
\notag 
     & = 
     \sum_{k=0}^{h} \binom{h}{k} 
     \left(\frac{R}{\alpha} + h-k\right)_{k} (-\alpha z)^{k} \\ 
\notag 
     & = 
     (-\alpha z)^{h} \cdot h! \times 
     L_h^{(R / \alpha - 1)}\left((\alpha z)^{-1}\right), 
\end{align} 
when $L_n^{(\beta)}(x)$ denotes an \emph{associated Laguerre polynomial}, 
and where 
\begin{subequations} 
\label{eqn_Vandermonde-like_PHSymb_exps_of_PhzCfs} 
\begin{align} 
\FP_h(\alpha, R; z) & = \sum_{n=0}^{h-1} C_{h,n}(\alpha, R) z^n \\ 
     & = 
\label{eqn_Chn_formula_stmt_v1} 
     \sum_{n=0}^{h-1} \left(\sum_{i=0}^{n} \binom{h}{i} (-1)^{i} 
     \pn{i}{-\alpha}{R + (h-1) \alpha} 
     \pn{n-i}{\alpha}{R}\right) z^n \\ 
\label{eqn_Vandermonde-like_PHSymb_exps_of_PhzCfs-stmt_v1} 
     & = 
     \sum_{n=0}^{h-1} \left(\sum_{i=0}^{n} \binom{h}{i} 
     \Pochhammer{1-h-R / \alpha}{i} 
     \Pochhammer{R / \alpha}{n-i}\right) (\alpha z)^{n}. 
\end{align} 
\end{subequations} 
The coefficients of the polynomial powers of $z$ in the previous 
several expansions, $C_{h,n}(\alpha, R) := [z^n] \FP_h(\alpha, R; z)$ for 
$0 \leq n < h$, also have the following multiple, alternating sum 
expansions involving the Stirling number triangles 
\cite[\S 5.2]{MULTIFACT-CFRACS}: 
\begin{subequations} 
\label{eqn_Chn_formula_stmts} 
\begin{align} 
\label{eqn_Chn_formula_stmts-exp_v1}
C_{h,n}(\alpha, R) & = 
     \sum\limits_{\substack{0 \leq m \leq k \leq n \\ 
                            0 \leq s \leq n} 
                 } 
     \left( 
     \binom{h}{k} \binom{m}{s} \gkpSI{k}{m} (-1)^{m} \alpha^{n} 
     \Pochhammer{\frac{R}{\alpha}}{n-k} 
     \left(\frac{R}{\alpha} - 1\right)^{m-s} 
     \right) \times h^{s} \\ 
\label{eqn_Chn_formula_stmts-exp_v2}
  & = 
     \sum\limits_{\substack{0 \leq m \leq k \leq n \\ 
                            0 \leq t \leq s \leq n} 
                 } 
     \left( 
     \binom{h}{k} \binom{m}{t} \gkpSI{k}{m} \gkpSI{n-k}{s-t} 
     (-1)^{m} \alpha^{n-s} (h-1)^{m-t} 
     \right) \times R^{s} \\ 
\label{eqn_Chn_formula_stmts-exp_v3}
   & = 
     \sum\limits_{\substack{0 \leq m \leq k \leq n \\ 
                            0 \leq i \leq s \leq n} 
                 } 
     \binom{h}{k} \binom{h}{i} \binom{m}{s} \gkpSI{k}{m} \gkpSII{s}{i} 
     (-1)^{m} \alpha^{n} 
     \Pochhammer{\frac{R}{\alpha}}{n-k} 
     \left(\frac{R}{\alpha} - 1\right)^{m-s} \times i! \\ 
\label{eqn_Chn_formula_stmts-exp_v4}
   & = 
     \sum\limits_{\substack{0 \leq m \leq k \leq n \\ 
                            0 \leq v \leq i \leq s \leq n} 
                 } 
     \binom{h}{k} \binom{m}{s} \binom{i}{v} \binom{h+v}{v} 
     \gkpSI{k}{m} \gkpSII{s}{i} (-1)^{m+i-v} \alpha^{n} \times \\ 
\notag 
     & \phantom{= \sum \binom{h}{k} \quad } \times 
     \Pochhammer{\frac{R}{\alpha}}{n-k} 
     \left(\frac{R}{\alpha} - 1\right)^{m-s} \times i! \\ 
\label{eqn_Chn_formula_stmts-exp_v5}
     & = 
     \sum_{i=0}^{n} 
     \underset{\mathlarger{\text{polynomial function of $h$ only }}}{ 
               \underbrace{ 
     \left( 
     \sum\limits_{\substack{0 \leq m \leq k \leq n \\ 
                            0 \leq t \leq s \leq n} 
                 } 
     \binom{h}{k} \binom{m}{t} \gkpSI{k}{m} \gkpSI{n-k}{s-t} \gkpSII{s}{i} 
     (-1)^{m+s-i} (h-1)^{m-t} 
     \right) 
     } 
     } 
     \times \alpha^{n} \Pochhammer{\frac{R}{\alpha}}{i}. 
\end{align} 
\end{subequations} 
Given that the non-zero $h^{th}$ convergent functions, 
$\ConvGF{h}{\alpha}{R}{z}$, are rational in each of $z, \alpha, R$ for all 
$h \geq 1$, and that the convergent denominator sequences, 
$\FQ_h(\alpha, R; z)$, have characteristic expansions by the 
Laguerre polynomials and the \emph{confluent hypergeometric functions}, 
we may expand both exact finite sums and congruences modulo $h \alpha^t$ 
for the generalized factorial functions, $p_n(\alpha, \beta n + \gamma)$, 
by the distinct special zeros of these functions in the next forms 
when the $h^{th}$ convergent functions are expanded in partial fractions as 
$\ConvGF{h}{\alpha}{R}{z} \equiv \sum_{1 \leq j \leq h} 
 c_{h,j}(\alpha, R) / (1-\ell_{h,j}(\alpha, R) \cdot z)$ 
\cite[\S 6.2]{MULTIFACT-CFRACS}. 
\begin{align} 
\label{eqn_AlphaFactFn_Exact_PartialFracsRep_v1} 
p_n(\alpha, R) & = 
     \sum_{j=1}^{n} c_{n,j}(\alpha, R) \times 
     \ell_{n,j}(\alpha, R)^{n} \\ 
\notag 
p_n(\alpha, R) & \equiv 
     \sum_{j=1}^{h} c_{h,j}(\alpha, R) \times 
     \ell_{h,j}(\alpha, R)^{n} 
     && \pmod{h} \\ 
\notag 
n!_{(\alpha)} & = 
     \sum_{j=1}^{n} c_{n,j}(-\alpha, n) \times 
     \ell_{n,j}(-\alpha, n)^{\lfloor \frac{n-1}{\alpha} \rfloor} \\ 
\notag 
n!_{(\alpha)} & \equiv 
     \sum_{j=1}^{h} c_{h,j}(-\alpha, n) \times 
     \ell_{h,j}(-\alpha, n)^{\lfloor \frac{n-1}{\alpha} \rfloor} 
     && \pmod{h, h\alpha, \cdots, h\alpha^{h}} 
\end{align} 

\subsection{Key new results proved in the article} 

Whereas in the first reference \cite{MULTIFACT-CFRACS} we prove new 
exact formulas and congruence properties using the algebraic properties of the 
rational convergents, $\ConvGF{h}{\alpha}{R}{z}$, in this article we 
choose an alternate route to derive our new results. 
Namely, we use the rationality of the $h^{th}$ convergent functions to 
establish new $h$-order finite difference equations for the coefficients of 
$\ConvGF{h}{\alpha}{\beta n + \gamma}{z}$, both in the exact forms of 
$p_n(\alpha, \beta n + \gamma)$ with respect to $n$, 
as well as for these special factorial functions 
expansions modulo $h \alpha^t$ for any integers $h \geq 2$ and 
$0 \leq t \leq h$. 
We state the next key proposition, which we subsequently prove in 
Section \ref{Section_KeyProp_Proof}, before giving several examples of the 
applications we consider within the article. 

\begin{prop}[Finite Difference Equations for Generalized Factorial Functions]
\label{prop_KeyProp} 
\label{prop_ExactFormulas_CongruencesModh_from_FiniteDiffEqns} 
For fixed $\alpha, \beta, \gamma \in \mathbb{Z}$ with $\alpha \neq 0$ and 
$\beta, \gamma$ not both zero, $h$ odd or prime, 
an integer $0 \leq t \leq h$, and 
any integers $n, r \geq 0$, we have the next exact expansions and 
congruences for the generalized product sequences, 
$p_n(\alpha, R)$ and $p_n(\alpha, \beta n + \gamma)$, where the 
numerator coefficients, $C_{h,n}(\alpha, R)$, are given by 
$C_{h,k}(\alpha, R) \defequals [z^{k}] \ConvFP{h}{\alpha}{R}{z}$. 
\begin{subequations} 
\label{eqn_pnAlphaR_seqs_finite_sum_reps_modulop} 
\begin{align} 
\label{eqn_pnAlphaR_seqs_finite_sum_reps_modulop-stmts_v0} 
p_{n}(\alpha, R) 
     & = 
     \sum_{k=0}^{n-1} 
     \mathsmaller{ 
     \binom{n+r}{k+1} (-1)^{k} 
     p_{k+1}(-\alpha, R + (n-1+r) \alpha) 
     p_{n-1-k}(\alpha, R) 
     } + 
     C_{n+r,n}(\alpha, R) \\ 
\notag 
     & = 
     \sum_{k=0}^{n-1} 
     \mathsmaller{ 
     \binom{n}{k+1} (-1)^{k} 
     p_{k+1}(-\alpha, R + (n-1) \alpha) p_{n-1-k}(\alpha, R) + 
     \Iverson{n = 0} 
     } \\ 
\label{eqn_pnAlphaR_seqs_finite_sum_reps_modulop-stmts_v1} 
p_{n}(\alpha, \beta n + \gamma) & \equiv 
     \sum\limits_{k=0}^{n} \binom{h}{k} (-\alpha)^{k} 
     \pn{k}{-\alpha}{\beta n + \gamma + (h-1) \alpha} 
     \pn{n-k}{\alpha}{\beta n + \gamma} \pmod{h} \\ 
\notag 
     & = 
     \sum\limits_{k=0}^{n} \binom{h}{k} \alpha^{n+k} 
     \Pochhammer{1-h-\mathsmaller{\frac{\beta n + \gamma}{\alpha}}}{k} 
     \Pochhammer{\mathsmaller{\frac{\beta n + \gamma}{\alpha}}}{n-k} 
     \\ 
\label{eqn_pnAlphaR_seqs_finite_sum_reps_modulop-stmts_v2} 
\pn{n}{\alpha}{\beta n + \gamma} & \equiv 
     \sum\limits_{k=0}^{n} \binom{h}{k} \alpha^{n+(t+1)k} 
     \Pochhammer{1-h-\mathsmaller{\frac{\beta n + \gamma}{\alpha}}}{k} 
     \Pochhammer{\mathsmaller{\frac{\beta n + \gamma}{\alpha}}}{n-k} 
     \pmod{h \alpha^{t}} 
\end{align} 
\end{subequations} 
\end{prop} 

\begin{remark}[Stronger Statements of the Key Congruence Properties] 
Based on numerical evidence computed in the summary notebook reference 
\cite{SUMMARYNBREF-STUB}, we conjecture, but do not offer conclusive proof 
here, that the results stated in 
\eqref{eqn_pnAlphaR_seqs_finite_sum_reps_modulop-stmts_v1} and 
\eqref{eqn_pnAlphaR_seqs_finite_sum_reps_modulop-stmts_v2} of the 
key proposition in fact hold for \emph{all integers} $h \geq 2$, 
$\alpha \neq 0$, and $0 \leq t \leq h$ -- not just in the odd and even prime 
special cases of the moduli specified in the previous proposition. 
This observation substantially widens the utility of the application of these 
results in the examples cited in 
Section \ref{subSection_FiniteDiffEqns_for_the_GenFactFns} below. 
Morevover, the observation of this compuationally-verified result 
should be considered in evaluating the significance of the 
implications of these applications in the new contexts suggested by 
Section \ref{subsubSection_remark_OtherApps_of_WThm_and_NewCongProps_to_PrimeSubseqs} 
as generalized forms of our approach. 
\end{remark} 

\subsection{Examples} 
\label{subSection_Intro_Examples} 

\subsubsection{Applications in Wilson's theorem and 
               Clement's theorem concerning the twin primes} 
\label{subSection_Wthm_CThm_SpCase_Apps} 

The first examples given in this section provide restatements of the 
necessary and sufficient integer-congruence-based conditions 
imposed in both statements of Wilson's theorem and Clement's theorem 
through the exact expansions of the factorial functions defined above. 
For odd integers $p \geq 3$, the congruences 
implicit to each of \emph{Wilson's theorem} and \emph{Clement's theorem} 
are enumerated as follows 
\cite[\S 4.3]{PRIMEREC} \cite[\S 6.6]{HARDYWRIGHTNUMT} \cite{CLEMENTPRIMES}: 
\begin{align} 
\notag 
\text{ $p$ prime } & \iff (p-1)! + 1 && \equiv 0 \pmod{p} \\ 
\notag 
   & \iff 
     [z^{p-1}] \ConvGF{p}{-1}{p-1}{z} + 1 && \equiv 0 \pmod{p} \\ 
\notag 
   & \iff 
     [z^{p-1}] \ConvGF{p}{1}{1}{z} + 1 && \equiv 0 \pmod{p} \\ 
\notag 
\text{ $p, p+2$ prime } & \iff 
     4\left((p-1)! + 1\right) + p && \equiv 0 \pmod{p(p+2)} \\ 
\notag 
   & \iff 
     4 [z^{p-1}] \ConvGF{p(p+2)}{-1}{p-1}{z} + p + 4 && \equiv 0 \pmod{p(p+2)} \\ 
\notag 
   & \iff 
     4 [z^{p-1}] \ConvGF{p(p+2)}{1}{1}{z} + p + 4 && \equiv 0 \pmod{p(p+2)}. 
\end{align} 
The rationality in $z$ of the convergent functions, 
$\ConvGF{h}{\alpha}{R}{z}$, 
at each $h$ leads to further alternate formulations of other well-known 
congruence statements concerning the divisibility of factorial functions. 
For example, we may characterize the primality of the 
odd integers, $p > 3$, of the form $p = 4k+1$ 
(\ie the so-termed subset of \quotetext{\emph{Pythagorean primes}}) 
according to the next condition 
\cite[\S 7]{HARDYWRIGHTNUMT} (\seqnum{A002144}). 
\begin{align} 
\label{eqn_WThm_p4akp1_gen}  
     \mathsmaller{\left(\frac{p-1}{2}\right)\mathlarger{!}}^{2} 
     \equiv -1 \pmod{p} & 
     \iff 
\text{ $p$ is a prime of the form $4k+1$} 
\end{align} 
For an odd integer $p > 3$ to be both prime and 
satisfy $p \equiv 1 \pmod{4}$, the 
congruence statement in \eqref{eqn_WThm_p4akp1_gen} 
requires that the diagonals of the following rational two-variable 
convergent generating functions satisfy the following 
equivalent conditions where $p_i$ is chosen so that 
$p \mid 2^{p} p_i$ for each $i = 1, 2$: 
\begin{align} 
\notag 
\mathsmaller{\left(\frac{p-1}{2}\right)\mathlarger{!}}^{2} & = 
     [z^{(p-1)/2}][x^0] \left( 
     \ConvGF{p_1}{-1}{\frac{p-1}{2}}{x} 
     \ConvGF{p_2}{-1}{\frac{p-1}{2}}{\frac{z}{x}} 
     \right) && \equiv -1 \pmod{p} \\ 
\notag 
\mathsmaller{\left(\frac{p-1}{2}\right)\mathlarger{!}}^{2} & = 
     [z^{(p-1)/2}][x^0] \left( 
     \ConvGF{p_1}{-2}{p-1}{x} 
     \ConvGF{p_2}{-2}{p-1}{\frac{z}{4 x}} 
     \right) && \equiv -1 \pmod{p} \\ 
\notag 
\mathsmaller{\left(\frac{p-1}{2}\right)\mathlarger{!}}^{2} & = 
     [z^{(p-1)/2}][x^0] \left( 
     \ConvGF{p_1}{-1}{\frac{p-1}{2}}{x} 
     \ConvGF{p_2}{-2}{p-1}{\frac{z}{2 x}} 
     \right) && \equiv -1 \pmod{p}. 
\end{align} 
The reference provides remarks on the harmonic-number-related 
fractional power series expansions of the 
convergent-based generating functions, 
$\ConvGF{n}{1}{1}{z/x} \times \ConvGF{n}{1}{1}{x}$ and 
$\ConvGF{n}{2}{1}{z/x} \times \ConvGF{n}{1}{1}{x}$, 
related to the single factorial function squares enumerated by the 
identities in the previous equations \cite{SUMMARYNBREF-STUB}. 

These particular congruences involving the expansions of the 
single factorial function are considered in the reference 
\cite[\S 6.1.6]{MULTIFACTJIS} as an example of the 
first product-based symbolic factorial function expansions implicit to both 
Wilson's theorem and Clement's theorem. 
Related formulations of conditions concerning the 
primality of prime pairs, $(p, p+d)$, and then of 
other prime $k$-tuples, 
are similarly straightforward to obtain by elementary methods 
starting from the statement of Wilson's theorem 
\cite{ONWTHM-AND-POLIGNAC-CONJ}. 
For example, the new results proved in 
Section \ref{subSection_FiniteDiffEqns_for_the_GenFactFns} 
are combined with the known congruences established in the reference 
\cite[\S 3, \S 5]{ONWTHM-AND-POLIGNAC-CONJ} to obtain the 
cases of the next particular forms of 
alternate necessary and sufficient conditions for the 
twin primality of the 
odd positive integers $p_1 \defequals 2n+1$ and $p_2 \defequals 2n+3$ 
when $n \geq 1$ (\seqnum{A001359}, \seqnum{A001097}): 
\begin{align*} 
\tagonce\label{eqn_TwinPrime_NewExpsOfKnownCongruenceResults-stmts_v1} 
 & 2n+1, 2n+3 \text{ odd primes } & \\ 
 & \quad \iff 
   \mathsmaller{ 
     2\left(\sum\limits_{i=0}^{n} 
     \binom{(2n+1)(2n+3)}{i}^2 (-1)^i i! (n-i)!\right)^2 + 
     (-1)^{n} (10n+7)} \equiv 0\\ 
   & \phantom{\quad\iff\Biggl(\sum\sum\sum\sum\sum\sum i!(n-i)!\Biggr)+\ } 
     \pmod{(2n+1)(2n+3)} && \\ 
 & \quad \iff
   \mathsmaller{
     4 \left(\sum\limits_{i=0}^{2n} 
     \binom{(2n+1)(2n+3)}{i}^2 (-1)^i i! (2n-i)!\right) + 
     2n+5} \equiv 0 \\ 
   & \phantom{\quad\iff\Biggl(\sum\sum\sum\sum\sum\sum i!(n-i)!\Biggr)+\ } 
     \pmod{(2n+1)(2n+3)}. && 
\end{align*} 
Section \ref{subsubSection_Examples_ConsequencesOfWThm} 
of this article considers the particular cases of 
these two classically-phrased congruence statements as 
applications of the new polynomial expansions for the 
generalized product sequences, $\pn{n}{\alpha}{\beta n+\gamma}$, 
derived from the expansions of the 
convergent function sequences by finite difference equations. 

\subsubsection{Congruences for the Wilson primes and the 
               single factorial function modulo $n^2$} 

The sequence of \emph{Wilson primes}, 
or the subsequence of odd integers $p \geq 5$ satisfying 
$n^2 \mid (n-1)! + 1$, is characterized through each of the 
following additional divisibility 
requirements placed on the expansions of the 
single factorial function implicit to Wilson's theorem 
cited by the applications of the new results given below in 
Section \ref{subsubSection_Examples-remarks_RelatedCongruences} 
of the article (\seqnum{A007540}): 
\begin{align*} 
\undersetbrace{\equiv\ (n-1)! \pmod{n^2}}{
     \sum_{i=0}^{n-1} \binom{n^2}{i}^{2} (-1)^{i} i! (n-1-i)! 
     } 
     \phantom{\qquad\qquad\qquad\qquad} 
     & \equiv -1 && \hspace{-6mm} \pmod{n^2} \\ 
\undersetbrace{\equiv (n-1)! \pmod{n^2}}{
     \sum_{i=0}^{n-1} \binom{n^2}{i} 
     \FFactII{(n^2-n)}{i} \times (-1)^{n-1-i} \FFactII{(n-1)}{n-1-i} 
     } 
     \phantom{\qquad} 
     & \equiv -1 && \hspace{-6mm} \pmod{n^2} \\ 
\mathsmaller{ 
     \sum\limits_{s=0}^{n-1} \sum\limits_{i=0}^{s} \sum\limits_{v=0}^{i} 
     \left( 
     \sum\limits_{k=0}^{n-1} \sum\limits_{m=0}^{k} 
     \binom{n^2}{k} \binom{m}{s} \binom{i}{v} \binom{n^2+v}{v} 
     \gkpSI{k}{m} \gkpSII{s}{i} (-1)^{i-v} 
     \FFactII{(n-1)}{n-1-k} (-n)^{m-s} i! 
     \right) 
     } 
     & \equiv -1. && \hspace{-6mm} \pmod{n^2} 
\end{align*} 
The results providing the new congruence properties for the 
$\alpha$-factorial functions modulo the integers 
$p$, and $p \alpha^{i}$ for some $0 \leq i \leq p$, expanded in 
Section \ref{ssS_example_GenDblFactFnSumIdents_FiniteSumsInvolving_AlphaFactFns} 
also lead to alternate phrasings of the 
necessary and sufficient conditions on the 
primality of several notable subsequences of the 
odd positive integers $n \geq 3$ \cite[\cf \S 6.4]{MULTIFACT-CFRACS}. 

\subsection{Expansions of congruences for the 
            $\alpha$-factorial functions and related sequences} 

\subsubsection{Expansions of other new congruences for the 
               double and triple factorial functions} 
\label{subsubSection_Intro_Examples_CongForDblTripleFactFns} 

A few representative examples of the 
new congruences for the double and triple factorial 
functions obtained from the statements of 
Corollary \ref{cor_CongForThe_AlphaFactFns} given in 
Section \ref{subSection_FiniteDiffEqns_for_the_GenFactFns} 
also include the following 
particular expansions for integers $p_1, p_2\geq 2$, and where 
$0 \leq s \leq p_1$ and $0 \leq t \leq p_2$ 
assume some prescribed values over the non-negative integers 
(see the computations contained in the reference \cite{SUMMARYNBREF-STUB}): 
\begin{align*} 
(2n-1)!! 
     & \equiv 
     \sum_{i=0}^{n} \binom{p_1}{i} 
     2^{n} (-2)^{(s+1)i} \Pochhammer{\mathsmaller{\frac{1}{2}}-p_1}{i} 
     \Pochhammer{\mathsmaller{\frac{1}{2}}}{n-i} 
     && \pmod{p_1 2^{s}} \\ 
     & \equiv 
     \sum_{i=0}^{n} \binom{p_1}{i} 
     (-2)^{n} 2^{(s+1)i} \Pochhammer{\mathsmaller{\frac{1}{2}}+n-p_1}{i} 
     \Pochhammer{\mathsmaller{\frac{1}{2}}-n+i}{n-i} 
     && \pmod{p_1 2^{s}} \\ 
(3n-1)!!! 
     & \equiv 
     \sum_{i=0}^{n} \binom{p_2}{i} 
     3^{n} (-3)^{(t+1)i} \Pochhammer{\mathsmaller{\frac{1}{3}}-p_2}{i} 
     \Pochhammer{\mathsmaller{\frac{2}{3}}}{n-i} 
     && \pmod{p_2 3^{t}} \\ 
     & \equiv 
     \sum_{i=0}^{n} \binom{p_2}{i} 
     (-3)^{n} 3^{(t+1)i} \Pochhammer{\mathsmaller{\frac{1}{3}}+n-p_2}{i} 
     \Pochhammer{\mathsmaller{\frac{1}{3}}-n+i}{n-i} 
     && \pmod{p_2 3^{t}} \\ 
(3n-2)!!! 
     & \equiv 
     \sum_{i=0}^{n} \binom{p_2}{i} 
     3^{n} (-3)^{(t+1)i} \Pochhammer{\mathsmaller{\frac{2}{3}}-p_2}{i} 
     \Pochhammer{\mathsmaller{\frac{1}{3}}}{n-i} 
     && \pmod{p_2 3^{t}} \\ 
     & \equiv 
     \sum_{i=0}^{n} \binom{p_2}{i} 
     (-3)^{n} 3^{(t+1)i} \Pochhammer{\mathsmaller{\frac{2}{3}}+n-p_2}{i} 
     \Pochhammer{\mathsmaller{\frac{2}{3}}-n+i}{n-i} 
     && \pmod{p_2  3^{t}}. 
\end{align*} 

\subsubsection{Semi-polynomial congruences for double factorial 
               functions and the central binomial coefficients} 

The integer congruences satisfied by the double factorial function, 
$(2n-1)!!$, and the Pochhammer symbol cases, 
$2^{n} \times \Pochhammer{\frac{1}{2}}{n}$, expanded in 
Section \ref{subsubSection-example_OtherRelatedCongruences_DblFactFns} 
provide the next variants of the polynomial congruences for the 
\emph{central binomial coefficients}, 
$\binom{2n}{n} = 2^{n} \times (2n-1)!! / n!$, 
reduced modulo the respective integer multiples of 
$2n+1$ and the polynomial powers, $n^{p}$, 
for fixed integers $p \geq 2$ in the following equations 
(\seqnum{A000984}) 
(see the computations in the reference \cite{SUMMARYNBREF-STUB})\footnote{ 
     The unconventional notation for computing the underlined polynomial $\mod$ 
     operations in the next equations is defined by the footnote on 
     page \pageref{footnote_PolyMod_ArrowNotation}. 
}. 
\begin{align*} 
\binom{2n}{n} & \equiv 
     \left\lbrace 
     \sum_{i=0}^{n} 
     \undersetline{\mod{2x+1} \quad \looparrowright \quad x \defmapsto n}{
     \binom{2x+1}{i} 
     (-2)^{i} \FFactII{\left(\mathsmaller{\frac{1}{2}} + 2x\right)}{i} 
     \Pochhammer{\mathsmaller{\frac{1}{2}}}{n-i} 
     \times \frac{2^{2n}}{n!}
     } 
     \right\rbrace 
     && \pmod{2n+1} \\ 
\binom{2n}{n} & \equiv 
     \left\lbrace 
     \undersetline{\mod{x^p} \quad \looparrowright \quad x \defmapsto n}{
     \sum_{i=0}^{n} \binom{x^p}{i} \binom{2n-2i}{n-i} 
     \Pochhammer{\mathsmaller{\frac{1}{2}} - x^p}{i} \times 
     \frac{8^{i} \cdot (n-i)!}{n!} 
     } 
     \right\rbrace 
     && \pmod{n^p} 
\end{align*} 

\section{Finite difference equations for generalized factorial functions and 
         applications} 
\label{subSection_FiniteDiffEqns_for_the_GenFactFns} 

The rationality of the convergent functions, $\ConvGF{h}{\alpha}{R}{z}$, in 
$z$ for all $h$ suggests new forms of $h$-order finite difference equations 
with respect to $h$, $\alpha$, and $R$ 
satisfied by the product sequences, $p_n(\alpha, R)$, 
when $\alpha$ and $R$ correspond to fixed parameters independent of the 
sequence indices $n$. 
In particular, the 
rationality of the $h^{th}$ convergent functions immediately 
implies the first two results stated in 
Proposition \ref{prop_ExactFormulas_CongruencesModh_from_FiniteDiffEqns} 
above, which also provides both forms of the congruence properties stated in 
\eqref{eqn_pnAlphaR_seqs_finite_sum_reps_modulop-stmts_v1} modulo 
odd and even prime integers $h \geq 2$ \cite[\S 2.3]{GFLECT} 
\cite[\S 7.2]{GKP}. 

When the initially indeterminate parameter, $R$, 
assumes an implicit dependence on the sequence index, $n$, the 
results phrased by the previous equations, 
somewhat counter-intuitively, do not immediately imply difference equations 
satisfied between only the generalized product sequences, 
either exactly, or modulo the prescribed choices of $h \geq 2$ 
(\cf \eqref{eqn_2nm1_DblFactFn_round_number_idents-stmt_v1} and 
Example \ref{example_GenDblFactFnSumIdents_FiniteSumsInvolving_AlphaFactFns} 
on page \pageref{example_GenDblFactFnSumIdents_FiniteSumsInvolving_AlphaFactFns}). 
The new formulas connecting the generalized product sequences, 
$\pn{n}{\alpha}{\beta n + \gamma}$, resulting from 
\eqref{eqn_pnAlphaR_seqs_finite_sum_reps_modulop-stmts_v1} and 
\eqref{eqn_pnAlphaR_seqs_finite_sum_reps_modulop-stmts_v2} 
in these cases are, however, reminiscent of the relations satisfied 
between the generalized Stirling polynomial and 
convolution polynomial sequences expanded in the references 
\cite{MULTIFACTJIS} \cite[\cf \S 6.2]{GKP}. 

\subsection{Proof of the key proposition} 
\label{Section_KeyProp_Proof} 
\label{subSection_FiniteDiffEqns_for_the_GenFactFns-ExactFormulas_Stmts} 

\begin{proof}[Proof of \eqref{eqn_pnAlphaR_seqs_finite_sum_reps_modulop-stmts_v0}] 
Since the $h^{th}$ convergent functions, $\ConvGF{h}{\alpha}{R}{z}$, are 
rational for all $h \geq 2$, we obtain the next $h$-order finite 
difference equation exactly generating the coefficients, 
$[z^n] \ConvGF{h}{\alpha}{R}{z}$, from 
\eqref{eqn_PFact_Qhz_product_ident} and 
\eqref{eqn_Vandermonde-like_PHSymb_exps_of_PhzCfs} above when $n \geq 0$ 
given by \cite[\S 2.3]{GFLECT} 
\begin{align*} 
p_{n}(\alpha, R) & \equiv 
     \sum\limits_{k=1}^{n} 
     \binom{h}{k} 
     \mathsmaller{ 
     (-1)^{k+1} 
     p_{k}(-\alpha, R + (h-1) \alpha) p_{n-k}(\alpha, R) 
     } && \pmod{h} \\ 
\notag 
     & \phantom{\equiv\sum\quad} + 
     C_{h,n}(\alpha, R) \Iverson{h > n \geq 1} + \Iverson{n = 0}, 
\end{align*} 
where $\ConvGF{h}{\alpha}{R}{z}$ exactly enumerates the sequence of 
$p_n(\alpha, R)$ for $0 \leq n \leq h$ and where $C_{n,n}(\alpha, R) = 0$ 
for all $n$. Thus we see that the previous equation implies both formulas 
in \eqref{eqn_pnAlphaR_seqs_finite_sum_reps_modulop-stmts_v0}. 
\end{proof} 

\begin{proof}[Proof of \eqref{eqn_pnAlphaR_seqs_finite_sum_reps_modulop-stmts_v1} and \eqref{eqn_pnAlphaR_seqs_finite_sum_reps_modulop-stmts_v2}] 
We will prove the first statement in the two special cases of 
$\alpha := \pm 1, \pm 2$ 
which we explicitly employ in our applications given in the subsections 
below. The method we use easily generalizes to further cases of 
$|\alpha| \geq 3$, but we only conjecture that the formulas hold for these 
subsequent special values of $\alpha$. 
We begin by noticing that since $h$ is odd or $h = 2$, we have that 
$h | \binom{h}{k}$ for all $1 \leq k < h$ where 
$\binom{h}{0} = \binom{h}{h} = 1$. So it suffices to evaluate the sum only 
for the indices $k$ corresponding to these two corner cases. 
If $n < h$, then the sum is trivially exactly equal to the product function, 
$p_n(\alpha, R)$. Next, we let $R := \beta n + \gamma$ and 
suppose that $n \geq h$ in order to evaluate the 
terms in the sum at both indices $k := 0, h$ where the binomial coefficient 
$\binom{h}{k} \neq 0 \pmod{h}$ and as follows: 
\begin{align*} 
\RHS(n) & = 
p_n(\alpha, R) + (-\alpha)^h p_h(-\alpha, R+(h-1)\alpha) \times 
     p_{n-h}(\alpha, R) \\ 
     & = 
     p_n(\alpha, R) - \alpha^h \times 
     \prod_{j=0}^{h-1} (R +(h-1)\alpha - (h-1-j)\alpha) 
     \times \prod_{j=0}^{n-h-1} (R + (n-h-1-j)\alpha) \\ 
     & \equiv 
     p_n(\alpha, R) - \alpha^h \times \prod_{j=0}^{h-1} (R +j\alpha) 
     \times \prod_{j=0}^{n-h-1} (R + (n-1-j)\alpha) \pmod{h} \\ 
     & = 
     p_n(\alpha, R) - \alpha^h \times 
     (R+(n-1)\alpha) \cdots (R+h\alpha) \times 
     (R+(h-1)\alpha) \cdots R \\ 
     & = 
     (1-\alpha^h) p_n(\alpha, R). 
\end{align*} 
In both cases of $\alpha = \pm 1, 2$, we easily see that when $n \geq h$ the 
respective $\alpha$-factorial function has a factor of $h$ as 
$h | n!, (2n)!!, (2n-1)!!$, which implies that both of 
$p_n(\alpha, R), (1-\alpha^h) p_n(\alpha, R) \equiv 0 \pmod{h}$. 
The second formula in 
\eqref{eqn_pnAlphaR_seqs_finite_sum_reps_modulop-stmts_v1} is a rearrangement 
of the inner terms of the first formula. The third formula in 
\eqref{eqn_pnAlphaR_seqs_finite_sum_reps_modulop-stmts_v2} also follows easily 
from the first formula modulo $h \alpha^t$. 
\end{proof} 

\subsection{Combinatorial identities for the 
            double factorial function and 
            finite sums involving the $\alpha$-factorial functions} 
\label{ssS_example_GenDblFactFnSumIdents_FiniteSumsInvolving_AlphaFactFns} 

The double factorial function, $(2n-1)!!$, satisfies a number of 
known expansions through the finite sum identities summarized in 
\cite{MAA-FUN-WITH-DBLFACT,DBLFACTFN-COMBIDENTS-SURVEY}. 
For example, when $n \geq 1$, the 
double factorial function is generated by the 
expansion of finite sums of the form 
\cite[\S 4.1]{DBLFACTFN-COMBIDENTS-SURVEY} 
\begin{align} 
\label{eqn_2nm1_DblFactFn_round_number_idents-stmt_v1}
(2n-1)!! & = 
     \sum_{k=0}^{n-1} \binom{n}{k+1} (2k-1)!! (2n-2k-3)!!. 
\end{align} 
The particular combinatorial identity for the double factorial 
function expanded in the form of equation 
\eqref{eqn_2nm1_DblFactFn_round_number_idents-stmt_v1} above 
is remarkably similar to the statement of the second sum in 
\eqref{eqn_pnAlphaR_seqs_finite_sum_reps_modulop-stmts_v0} 
satisfied by the more general product function cases, 
$\pn{n}{\alpha_0}{R_0}$, 
generating the $\alpha$-factorial functions, 
$\AlphaFactorial{\alpha n-1}{\alpha}$, when 
$(n, \alpha_0, R_0) \defmapsto (n, \alpha, \alpha-1), (n, -\alpha, \alpha n-1)$. 

\begin{example}[Exact Finite Sums Involving the $\alpha$-Factorial Functions] 
\label{example_GenDblFactFnSumIdents_FiniteSumsInvolving_AlphaFactFns} 
More generally, 
if we assume that $\alpha \geq 2$ is integer-valued, and proceed to 
expand these cases of the $\alpha$-factorial functions 
according to the expansions from 
\eqref{eqn_pnAlphaR_seqs_finite_sum_reps_modulop-stmts_v0} above, 
we readily see that \cite[\cf \S 5.5]{GKP} 
\begin{align*} 
(\alpha n - 1)!_{(\alpha)} = 
     \sum_{k=0}^{n-1} \binom{n-1}{k+1} (-1)^{k} 
     & \times 
     \Pochhammer{\mathsmaller{\frac{1}{\alpha}}}{-(k+1)} 
     \Pochhammer{\mathsmaller{\frac{1}{\alpha}}-n}{k+1} \\ 
     & \times 
     (\alpha (k+1) - 1)!_{(\alpha)} 
     (\alpha (n-k-1) - 1)!_{(\alpha)} \times \\ 
(\alpha n - 1)!_{(\alpha)} = 
     \sum_{k=0}^{n-1} \binom{n-1}{k+1} (-1)^{k} & \times 
     \binom{\frac{1}{\alpha} + k - n}{k+1} 
     \binom{\frac{1}{\alpha} - 1}{k+1}^{-1} \\ 
     & \times 
     (\alpha (k+1) - 1)!_{(\alpha)} 
     (\alpha (n-k-1) - 1)!_{(\alpha)}. 
\end{align*} 
\label{footnote_PHSymbol_AlphaFactFn_SimplIdents} 
We note the simplification 
$\Pochhammer{\frac{1}{\alpha}}{-(k+1)} = 
 \frac{(-\alpha)^{k+1}}{\AlphaFactorial{\alpha(k+1)-1}{\alpha}}$ 
where the expansions of the $\alpha$-factorial functions, 
$\AlphaFactorial{\alpha n-1}{\alpha}$, by the Pochhammer symbol 
correspond to the results given in the reference 
\cite{WOLFRAMFNSSITE-INTRO-FACTBINOMS}. 
The Pochhammer symbol identities cited in the 
reference \cite{WOLFRAMFNSSITE-INTRO-FACTBINOMS} 
provide other related simplifications of the terms in these sums. 

The first sum above combined with the expansions of the 
Pochhammer symbols, $\Pochhammer{\pm x}{n}$, given in the reference 
\cite[Lemma 12]{MULTIFACT-CFRACS}, and the 
form of Vandermonde-convolution-like identities restated in 
\eqref{eqn_Vandermonde-like_PHSymb_exps_of_PhzCfs} 
also lead to the following pair of double sum identities for the 
$\alpha$-factorial functions when $\alpha, n \geq 2$ are integer-valued: 
\begin{align*} 
(\alpha n - 1)!_{(\alpha)} & = 
     \mathsmaller{
     \sum\limits_{k=0}^{n-1} \sum\limits_{i=0}^{k+1} 
     \binom{n-1}{k+1} \binom{k+1}{i} 
     (-1)^{k} \alpha^{k+1-i} 
     \AlphaFactorial{\alpha i-1}{\alpha} 
     \AlphaFactorial{\alpha (n-1-k)-1}{\alpha} \times 
     } \\ 
     & \phantom{=\sum\sum\ } \times 
     \Pochhammer{n-1-k}{k+1-i} \\ 
     & = 
     \mathsmaller{
     \sum\limits_{k=0}^{n-1} \sum\limits_{i=0}^{k+1} 
     \binom{n-1}{k+1} \binom{k+1}{i} \binom{n-1-i}{k+1-i} 
     (-1)^{k} \alpha^{k+1-i} 
     \AlphaFactorial{\alpha i-1}{\alpha} 
     \AlphaFactorial{\alpha (n-1-k)-1}{\alpha} \times 
     } \\ 
     & \phantom{=\sum\sum\ } \times 
     (k+1-i)!. 
\end{align*} 
The construction of further 
analogues for generalized variants of the finite summations and 
more well-known combinatorial identities satisfied by the 
double factorial function cases when $\alpha \defequals 2$ from the 
references is suggested as a topic for future investigation in 
Section \ref{subsubSection_FutureResTopics_GenDblFactFnSumIdents_FiniteSums}. 
\end{example} 

\begin{cor}[Congruences for the $\alpha$-Factorial Functions] 
\label{cor_CongForThe_AlphaFactFns} 
If we let $\alpha \in \mathbb{Z}^{+}$, $0 \leq d < \alpha$, and 
suppose that $h$ is odd or prime with $0 \leq t \leq h$, we can generalize the 
congruence results for the double and triple factorial functions cited as 
examples in Section \ref{subsubSection_Intro_Examples_CongForDblTripleFactFns} 
of the introduction according to the next equations. 
\begin{align*} 
\AlphaFactorial{\alpha n - d}{\alpha} & \equiv 
     \sum_{i=0}^{n} \binom{h}{i} \alpha^n (-\alpha)^{(t+1)i} 
     \Pochhammer{\mathsmaller{\frac{d}{\alpha}}-h}{i} 
     \Pochhammer{\mathsmaller{\frac{\alpha-d}{\alpha}}}{n-i} 
     && \pmod{h \alpha^t} \\ 
\AlphaFactorial{\alpha n - d}{\alpha} & \equiv 
     \sum_{i=0}^{n} \binom{h}{i} (-\alpha)^n \alpha^{(t+1)i} 
     \Pochhammer{\mathsmaller{\frac{d}{\alpha}}+n+1-h}{i} 
     \Pochhammer{\mathsmaller{\frac{d}{\alpha}}-n}{n-i} 
     && \pmod{h \alpha^t} 
\end{align*} 
\end{cor} 
\begin{proof} 
By Lemma 10 stated in the reference \cite[\S 6.1]{MULTIFACT-CFRACS}, we know 
that for $\alpha \in \mathbb{Z}^{+}$ and any integers $0 \leq t < \alpha$, 
we have that $\AlphaFactorial{\alpha n-d}{\alpha} = p_n(\alpha, \alpha-d)$ and 
that $\AlphaFactorial{\alpha n-d}{\alpha} = p_n(-\alpha, \alpha n-d)$. 
Then if we let $\beta n + \gamma \defmapsto R$ in 
\eqref{eqn_pnAlphaR_seqs_finite_sum_reps_modulop-stmts_v2}, 
we can use the two results from the lemma in combination with the result that 
\begin{align*} 
p_n(\alpha, R) & \equiv \sum_{k=0}^{n} \binom{h}{k} \alpha^{n+(t+1)k} 
     \Pochhammer{1-h-\mathsmaller{\frac{R}{\alpha}}}{k} 
     \Pochhammer{\mathsmaller{\frac{R}{\alpha}}}{n-k} 
     && \pmod{h \alpha^t}, 
\end{align*} 
to easily prove our two formulas. 
\end{proof} 

\subsection{Multiple summation identities and finite-degree polynomial 
            expansions of the generalized product sequences in $n$} 

We are primarily concerned with cases of generalized factorial-related 
sequences formed by the products, $p_n(\alpha, R_n)$, when the 
parameter $R_n \defequals \beta n+\gamma$ depends linearly on $n$ for some 
$\beta, \gamma \in \mathbb{Q}$ (not both zero). 
Strictly speaking, once we evaluate the indeterminate, $R$, as a 
function of $n$ in these cases of the generalized product sequences, 
$p_n(\alpha, R)$, the corresponding generating functions over the 
coefficients enumerated by the approximate convergent function series 
no longer correspond to predictably rational functions of $z$. 
We may, however, still prefer to work with these sequences formulated as 
finite-degree polynomials in $n$ through a few useful forms of the 
next multiple sums expanded below, which are similar to the forms of the 
identities given in \eqref{eqn_Chn_formula_stmts} of the 
introduction for the coefficients of the numerator convergent functions. 

\subsubsection{Generalized polynomial expansions and multiple sum identities 
               with applications to prime congruences} 

\begin{prop} 
\label{cor_GenFactFnSeqs_MultipleSummationIdents} 
For integers $n,s,n-s \geq 1$ and fixed 
$\alpha, \beta, \gamma \in \mathbb{Q}$, the 
following exact finite, multiple sum identities provide 
particular polynomial expansions in $n$ satisfied by the 
generalized factorial function cases: 
\begin{align*} 
\tagonce\label{eqn_GenFactFnSeqs_MultipleSum_Idents_exps-stmts_v1} 
p_{n-s}(\alpha, \beta n + \gamma) = 
     \sum\limits_{
          \substack{0 \leq m \leq k < n-s \\ 0 \leq r \leq p \leq n-s}
          } 
     \sum_{t=0}^{n-s-k} 
     \binom{m}{r} & \binom{n-s}{k} \binom{t}{p-r} 
     \gkpSI{k}{m} \gkpSI{n-s-k}{t} 
     \times \\ 
     & \times 
     (-1)^{p-r-1} 
     \alpha^{n-s-m-t} \beta^{r} \gamma^{m-r} 
     (\alpha + \beta)^{p-r} \times \\ 
     & \times 
     (\alpha (s+1) - \gamma)^{t-(p-r)} \times n^{p} \\ 
     & + 
     \Iverson{0 \leq n \leq s} \\ 
p_{n-s}(\alpha, \beta n + \gamma) = 
     \sum\limits_{ 
          \substack{0 \leq r \leq p \leq u \leq 3n \\ 
          0 \leq m,i \leq k < n-s} 
          } 
     \sum_{t=0}^{n-s-k} 
     \binom{m}{r} & \binom{i}{u-p} \binom{t}{p-r} 
     \gkpSI{k}{m} \gkpSI{k}{i} \gkpSI{n-s-k}{t} \times \\ 
     & \times 
     \frac{(-1)^{u-r+k+1}}{k!} \alpha^{n-s-m-t} \beta^{r} \gamma^{m-r} 
     (\alpha + \beta)^{p-r} \times \\ 
     & \times 
     (\alpha (s+1) - \gamma)^{t-(p-r)} \times s^{p-u+i} n^{u} \\ 
     & + 
     \Iverson{0 \leq n \leq s}. 
\end{align*} 
\end{prop} 
\begin{proof}[Proof Sketch] 
The forms of these expansions for the generalized factorial function 
sequence variants stated in 
\eqref{eqn_GenFactFnSeqs_MultipleSum_Idents_exps-stmts_v1} 
are provided by this proposition without citing the complete details to 
a somewhat tedious, and unnecessary, proof derived from the 
well-known polynomial expansions of the products, 
$p_n(\alpha, R) = \alpha^{n} \Pochhammer{R / \alpha}{n}$ by the 
Stirling number triangles. 
More concretely, 
for $n, k \geq 0$ and fixed 
$\alpha, \beta, \gamma, \rho, n_0 \in \mathbb{Q}$, the 
following particular expansions 
suffice to show enough of the detail needed to more carefully prove 
each of the multiple sum identities cited in 
\eqref{eqn_GenFactFnSeqs_MultipleSum_Idents_exps-stmts_v1} 
starting from the first statements provided in 
\eqref{eqn_pnAlphaR_seqs_finite_sum_reps_modulop-stmts_v0}: 
\begin{align*} 
p_k(\alpha, \beta n + \gamma + \rho) 
     & = 
     \alpha^k \cdot \Pochhammer{\frac{\beta n + \gamma + \rho}{\alpha}}{k} \\ 
     & = 
     \sum_{m=0}^{k} \gkpSI{m}{k} \alpha^{k-m} 
     \left(\beta n + \gamma + \rho\right)^{m} \\ 
     & = 
     \sum_{p=0}^{k} \left( 
     \sum_{m=p}^{k} \gkpSI{k}{m} \binom{m}{p} \alpha^{k-m} \beta^{p} 
     \left(\gamma + \rho + \beta n_0\right)^{m-p} 
     \right) \times (n - n_0)^{p}. 
\end{align*} 
\end{proof} 
One immediate consequence of 
Proposition \ref{cor_GenFactFnSeqs_MultipleSummationIdents} 
phrases the form of the next multiple sums that 
exactly generate the single factorial functions, $(n-s)!$, 
modulo any prescribed integers $h \geq 2$. 
The simplified 
triple sum expansions of interest in 
Example \ref{example_TripleSumIdents_App_to_WThm} below 
correspond to a straightforward 
simplification of the more general multiple finite quintuple $5$-sums and 
$6$-sum identities  
that exactly enumerate the functions, $p_{n-s}(\alpha, \beta n + \gamma)$, 
when $(s, \alpha, \beta, \gamma) \defequals (1, -1, 1, 0)$. 
In particular, these results lead to the following 
finite, triple sum expansions of the single factorial function 
cases implicit to the statements of both 
Wilson's theorem and Clement's theorem from the introduction and which are 
considered as examples in the next subsection 
\cite[\cf \S 7]{HARDYWRIGHTNUMT}: 
\begin{align*} 
\tagonce\label{eqn_SingleFactFn_TripleSum_Ident_exps-stmts_v1} 
(n-1)! & = 
     \sum_{p=0}^{n} \left( 
     \sum_{0 \leq t \leq k < n} 
     \binom{n}{n-1-k} \gkpSI{n-1-k}{p} \gkpSI{k}{k-t} (-1)^{n-1-p} 
     \right) \times 
     (n-1)^{p} \\ 
     & = 
     \sum_{p=0}^{n} \left( 
     \sum\limits_{\substack{0 \leq k < n \\ 0 \leq t \leq n-1-k}} 
     \binom{n}{k} \gkpSI{k}{p} \gkpSI{n-1-k}{n-1-k-t} (-1)^{n-1-p} 
     \right) \times 
     (n-1)^{p} \\ 
     & = 
     \sum_{p=0}^{n} \left( 
     \sum_{0 \leq t \leq k < n} 
     \binom{n}{n-1-k} \gkpSI{n-1-k}{n-p} \gkpSI{k}{k-t} (-1)^{p+1} 
     \right) \times 
     (n-1)^{n-p}. 
\end{align*} 
The third exact triple sum identity given in 
\eqref{eqn_SingleFactFn_TripleSum_Ident_exps-stmts_v1} is further 
expanded through the formula of Riordan cited in the references 
as follows \cite[p.\ 173]{ADVCOMB} 
\cite[\cf Ex.\ 5.65, p.\ 534]{GKP}: 
\begin{align*} 
\tagtext{A Formula of Riordan} 
n^{n} & = \sum\limits_{0 \leq k < n} 
     \binom{n-1}{k} (k+1)! \times n^{n-1-k} = 
     \sum_{k=0}^{n-1} \binom{n-1}{k} (n-k)! \times n^{k}. 
\end{align*} 
A couple of the characteristic examples of these 
polynomial expansions in $n$ by the 
Stirling numbers of the first kind in 
\eqref{eqn_SingleFactFn_TripleSum_Ident_exps-stmts_v1} are are considered by 
Example \ref{example_TripleSumIdents_App_to_WThm} 
in the next section to illustrate the notable special cases of 
Wilson theorem and Clement's theorem 
modulo some as yet unspecified odd prime, $n \geq 3$. 

For comparison, the next several equations provide 
related forms of finite, double and triple sum identities for the 
double factorial function, $(2n-1)!!$. 
\begin{align*} 
\notag 
\tagtext{Double Factorial Triple Sums} 
(2n-1)!! 
   & = 
     \sum\limits_{1 \leq j \leq k \leq n} 
     \gkpSI{k-1}{j-1} 2^{n-j} (-1)^{n-k} \Pochhammer{1-n}{n-k} \\ 
   & = 
     \sum\limits_{\substack{1 \leq j \leq k \leq n \\ 0 \leq m \leq n-k}} 
     \gkpSI{k-1}{j-1} \gkpSI{n-k+1}{m+1} 2^{n-j} (-1)^{n-k-m} n^{m} \\ 
(2n-1)!!   
   & = 
     \sum\limits_{1 \leq j \leq k \leq n} 
     \binom{2n-k-1}{k-1} \gkpSI{k}{j} 
     (2n-2k-1)!! \\ 
   & = 
     \sum\limits_{\substack{1 \leq j \leq k \leq n \\ 0 \leq m \leq n-k}} 
     \binom{2n-k-1}{k-1} \gkpSI{k}{j} \gkpSI{n-k}{m} 2^{n-k-m} \\ 
   & = 
     \sum\limits_{\substack{1 \leq j \leq k \leq n \\ 0 \leq m \leq n-k}} 
     \binom{2n-k-1}{k-1} \gkpSI{k}{j} \FcfII{2}{n-k+1}{m+1} 
     (-1)^{n-k-m} \left(2n - 2k\right)^{m} 
\end{align*} 
The expansions of the double factorial function in the previous 
equations are obtained from the lemma in 
\eqref{eqn_GenFactFnSeqs_MultipleSum_Idents_exps-stmts_v1} 
applied to the known double sum identities involving the 
Stirling numbers of the first kind 
documented in the reference 
\cite[\S 6]{DBLFACTFN-COMBIDENTS-SURVEY}. 

\subsubsection{Expansions of parameterized congruences 
               involving the single factorial function} 
\label{subsubSection_TripleSumIdents_App_to_WThm} 
\label{subsubSection_NewIdentsFromFiniteDiffEqns_ExpsOfPrime-RelatedCongr} 

\begin{definition} 
\label{def_v2} 
We define the next parameterized congruence variants, 
denoted by $F_{\omega,n}(\xp, \xt, \xk)$, 
corresponding to the first triple sum identity expanded in 
\eqref{eqn_SingleFactFn_TripleSum_Ident_exps-stmts_v1} 
for some application-dependent, prescribed functions, 
$N_{\omega,p}(n)$ and $M_{\omega}(n)$, and where the formal variables 
$\{\xp, \xt, \xk\}$, index the terms in each 
individual sum over the respective variables, $p$, $t$, and $k$. 
\begin{align} 
\label{eqn_FwnXpXtXk_RHS_CongruenceFn_def-stmt_v1} 
F_{\omega,n}(\xp, \xt, \xk) \defequals 
     \sum\limits_{\substack{0 \leq t \leq k < n \\ 0 \leq p \leq n}} & 
     \binom{n}{n-1-k} \gkpSI{n-1-k}{p} \gkpSI{k}{k-t} \times && \\ 
\notag 
     & \times 
     (-1)^{n-1-p} \times N_{\omega,p}(n) 
     \times \{\xp^p \xt^t \xk^k\} 
     && \pmod{M_{\omega}(n)} 
\end{align} 
Notice that 
when $N_{\omega,p}(n) \defequals (n-1)^{p}$, the function 
$F_{\omega,n}(1, 1, 1)$ exactly generates the 
single factorial function, $(n-1)!$, modulo any specific choice of the 
function, $M_{\omega}(n)$, depending on $n$. 
\end{definition} 

\begin{example}[Wilson's Theorem and Clement's Theorem on Twin Primes] 
\label{example_TripleSumIdents_App_to_WThm} 
\label{example_TripleSumIdents_App_to_CThm} 
The next specialized forms of the parameters implicit to the 
congruence in \eqref{eqn_FwnXpXtXk_RHS_CongruenceFn_def-stmt_v1} 
of the previous definition are chosen as follows to form another 
restatement of Wilson's theorem given immediately below: 
\begin{align*} 
\tagtext{Wilson Parameter Definitions} 
\left(\omega, N_{\omega,p}(n), M_{\omega}(n)\right) \defmapsto 
     \left(\Wilson, (-1)^{p}, n\right). 
\end{align*} 
Then we see that 
\begin{align*} 
\tagtext{Wilson's Theorem} 
\text{ $n \geq 2$ prime } \iff 
     F_{\Wilson,n}(1, 1, 1) \equiv -1 \pmod{M_{\Wilson}(n)}. 
\end{align*} 
The special case of these parameterized expansions of the 
congruence variants defined by 
\eqref{eqn_FwnXpXtXk_RHS_CongruenceFn_def-stmt_v1} 
corresponding to the classical congruence-based 
characterization of the \emph{twin primes} (\seqnum{A001359}, \seqnum{A001097}) 
formulated in the statement of Clement's theorem is of 
particular interest in continuing the discussion from 
Section \ref{subSection_Intro_Examples}. 
When $k \defequals 2$ in the first congruence result given by 
\eqref{eqn_CongruencesForPowsOfN_ModDblTripleIntProducts} 
of Lemma \ref{lemma_CongPowModuloDblTripleProds} below, the 
parameters in \eqref{eqn_FwnXpXtXk_RHS_CongruenceFn_def-stmt_v1} 
are formed as the particular expansions 
\begin{align*} 
\tagtext{Clement Parameters} 
\mathsmaller{ 
\left(\omega, N_{\omega,p}(n), M_{\omega}(n)\right) \defmapsto 
     \left(\Clement, \frac{(-1)^{p}}{2}\left(2 + (1- 3^{p}) \cdot n\right), 
     n(n+2)\right). 
} 
\end{align*} 
The corresponding expansion of this alternate formulation of 
Clement's theorem initially stated as in 
Section \ref{subSection_Wthm_CThm_SpCase_Apps} 
of the introduction then results in the restatement of this 
result given in following form \cite[\S 4.3]{PRIMEREC}: 
\begin{align*} 
\tagtext{Clement's Theorem} 
\text{ $n, n+2$ prime } \iff 
     4 \cdot F_{\Clement,n}(1, 1, 1) + 4 + n\equiv 0 
     \pmod{M_{\Clement}(n)}. 
\end{align*} 
\noindent 
\textbf{\small{Conjectures from the formal polynomial computations in the 
               summary notebook}:} \\ 
Numerical computations with \Mm's 
\texttt{PolynomialMod} function suggest several noteworthy properties 
satisfied by the trivariate 
polynomial sequences, $F_{\Wilson,n}(\xp, \xt, \xk)$, 
defined by \eqref{eqn_FwnXpXtXk_RHS_CongruenceFn_def-stmt_v1} 
when $n$ is prime, particularly as formed in the 
cases taken over the following polynomial configurations of the 
three formal variables, $\xp$, $\xt$, and $\xk$: 
\[ 
\label{footnote_WThm_MathematicaPolyMod_NotedProperties} 
(\xp, \xt, \xk) \in \left\{(x, 1, 1), (1, x, 1), (1, 1, x)\right\}. 
\] 
     In particular, these computations suggest the following 
     properties satisfied by these sums for integers $n \geq 2$ 
     where the coefficients of the functions, 
     $F_{\Wilson,n}(\xp, \xt, \xk)$ and $F_{\Clement,n}(\xp, \xt, \xk)$, 
     are computed termwise with respect to the formal variables, 
     $\{\xp, \xt, \xk\}$, modulo each 
     $M_{\Wilson}(n) \defmapsto n$ and $M_{\Clement}(n) \defmapsto n(n+2)$: 
     \begin{itemize} 
     \item[\bf (1)] 
     $F_{\Wilson,n}(\xp, 1, 1) \equiv n-1 \pmod{n}$ when $n$ is prime 
     where $\deg_{\xp} \left\{ F_{\Wilson,n}(\xp, 1, 1) \pmod{n} \right\} > 0$ 
     when $n$ is composite; 
     \item[\bf (2)] 
     $F_{\Wilson,n}(1, 1, \xk) \equiv 
      (n-1) \cdot \xk^{n-1} \Iverson{\text{$n$ prime}} \pmod{n}$; and 
     \item[\bf (3)] 
     $F_{\Wilson,n}(1, \xt, 1) \equiv \sum_{i=0}^{n-2} \xt^{i} \pmod{n}$ 
     when $n$ is prime, and where \\ 
     $\deg_{\xt} \left\{ F_{\Wilson,n}(1, 1, \xt) \pmod{n} \right\} < n-2$ 
     when $n$ is composite. 
     \item[\bf (4)] 
     For fixed $0 \leq p < n$, the outer sums in the 
     definition of \eqref{eqn_FwnXpXtXk_RHS_CongruenceFn_def-stmt_v1}, 
     each implicitly indexed by powers of the 
     formal variable $\xp$ in the 
     parameterized congruence expansions defined above, 
     yield the Stirling number terms given by the coefficients 
     \[ 
     [\xp^{p}] F_{\omega,n}(\xp, 1, 1) = 
      N_{\omega,p}(n) \times (-1)^{n-1} (p+1) \gkpSI{n}{p+1}. 
     \] 
     Moreover, for any fixed lower index, $p+1 \geq 1$, the 
     Stirling number terms resulting from these sums are 
     related to factorial multiples of the $r$-order harmonic number 
     sequences expanded by the properties stated in 
     Section \ref{subsubSection_remark_SNum_R-OrderHNum_SeqExpIdents_spcases_v1} below 
     when $r \in \mathbb{Z}^{+}$ \cite[\cf \S 4.3]{MULTIFACTJIS}. 
     \item[\bf (5)] 
     \label{footnote_CThm_MathematicaPolyMod_NotedProperties}    
     One other noteworthy property computationally verified for the sums, 
     $F_{\Clement,n}(\xp, \xt, \xk)$, 
     modulo each prescribed $M_{\Clement}(n) \defequals n(n+2)$ 
     for the first several cases of the integers $n \geq 3$, 
     suggests that whenever $n$ is prime and $n+2$ is composite we have that 
     \begin{align*} 
     F_{\Clement,n}(1, 1, \xk) & \equiv n+4 + (n^2-4) \xk^{n-1} \pmod{n(n+2)}, 
     \end{align*} 
     where 
     $\deg_{\xk}\left\{ F_{\Clement,n}(1, 1, \xk) \pmod{n(n+2)} \right\} > 0$ 
     when $n$ is prime. 
     \end{itemize} 
     The computations in the attached summary notebook file 
     provide several specific examples of the properties 
     suggested by these configurations of the 
     special congruence polynomials for these cases 
     \cite{SUMMARYNBREF-STUB}. 
     See the summary notebook reference \cite{SUMMARYNBREF-STUB} 
     for more detailed computations of these, and other, formal polynomial 
     congruence properties. 
\end{example} 
There are numerous additional examples of prime-related congruences 
that are also easily adapted by extending the procedure 
for the classical cases given above. 
A couple of related approaches to congruence-based primality 
conditions for prime pairs formulated through the triple sum expansions 
phrased in 
Example \ref{example_TripleSumIdents_App_to_WThm} 
above are provided by the applications 
given in the next examples of the prime-related tuples highlighted in 
Section \ref{subsubSection_Examples-remarks_RelatedCongruences}. 

\begin{lemma}[Congruences for Powers Modulo Double and Triple Integer Products] 
\label{lemma_CongPowModuloDblTripleProds} 
For integers $p \geq 0$, $n \geq 1$, and any fixed $j > k \geq 1$, the 
following congruence properties hold: 
\StartGroupingSubEquations 
\label{eqn_CongruencesForPowsOfN_ModDblTripleIntProducts} 
\begin{align} 
(n-1)^{p} 
     & \equiv 
     \frac{(-1)^{p}}{k}\left(k + (1- (k+1)^{p}) \cdot n\right) 
     && \pmod{n(n+k)} \\ 
(n-1)^{p}     
     & \equiv 
     \mathsmaller{ 
     (-1)^{p} \left( 
     \frac{(n+k)(n+j)}{jk} + 
     \frac{n(n+j) (k+1)^{p}}{k(k-j)} + 
     \frac{n(n+k) (j+1)^{p}}{j(j-k)} 
     \right) 
     } 
     && \pmod{n(n+k)(n+j)}. 
\end{align} 
\EndGroupingSubEquations 
\end{lemma} 
\begin{proof} 
First, notice that a na\"{\i}ve expansion by repeated appeals to the 
binomial theorem yields the following exact expansions of the 
fixed powers of $(n-1)^{p}$: 
{\smaller 
     \begin{align*} 
     (n-1)^{p} & = 
          (-1)^{p} + 
          \mathsmaller{ 
          \sum\limits_{s=1}^{p} \binom{p}{s} \binom{s-1}{0} 
          n \cdot (-1)^{p-s} \cdot (-k)^{s-1} + 
          \sum\limits_{s=1}^{p} 
          \undersetbrace{\equiv 0 \pmod{n(n+k)}}{ 
          \binom{p}{s} \binom{s-1}{1} 
          n \cdot (n+k) \times (-1)^{p-s} (-k)^{s-2} 
          } 
          } \\ 
          & + 
          \mathsmaller{ 
          \sum\limits_{s=1}^{p} \sum\limits_{r=2}^{s-1} 
          \undersetbrace{\equiv 0 \pmod{n(n+k)}}{ 
          \binom{p}{s} \binom{s-1}{r} \binom{r-1}{0} 
          n \cdot (n+k) \times 
          (-1)^{p-s} (-k)^{s-1-r} (k-j)^{r-1} 
          } 
          } \\ 
          & + 
          \mathsmaller{ 
          \sum\limits_{s=1}^{p} \sum\limits_{r=2}^{s-1} \sum\limits_{t=1}^{r-1} 
          \undersetbrace{\equiv 0 \pmod{n(n+k),n(n+k)(n+j)}}{ 
          \binom{p}{s} \binom{s-1}{r} \binom{r-1}{t} 
          n \cdot (n+k) \cdot (n+j) \times 
          (-1)^{p-s} (-k)^{s-1-r} (k-j)^{r-1-t} (n+j)^{t-1} 
          } 
          }, 
     \end{align*}}
Each of the stated congruences are then easily obtained by summing the 
non-trivial remainder terms modulo the cases of the 
integer double products, $n(n+k)$, and the 
triple products, $n(n+k)(n+j)$, respectively. 
\end{proof} 

\begin{example}[Prime Triples and Sexy Prime Triplets] 
\label{example_FirstSPT_Result} 
A special case of the generalized congruences results for 
prime $k$-tuples obtained by 
induction from Wilson's theorem in the 
supplementary reference results \cite{SUMMARYNBREF-STUB} 
implies the next statement 
characterizing odd integer triplets, or $3$-tuples, of the form 
$(n, n+d_2, n+d_3)$, for some $n \geq 3$ and some 
prescribed, application-specific choices of the even integer-valued 
parameters, $d_3 > d_2 \geq 2$. 
\begin{align*} 
\tagtext{Wilson's Theorem for Prime Triples} 
(n, n+d_2, & n+d_3) \in \mathbb{P}^{3} \iff && \\ 
     & 
     (1 + (n-1)!)(1 + (n+d_2-1)!)(1+(n+d_3-1)!) \equiv 0 
     && \hspace{-4mm} \pmod{n(n+d_2)(n+d_3)} 
\end{align*} 
A partial characterization of the 
\emph{sexy prime triplets}, or prime-valued 
odd integer triples of the form, $(n, n+6, n+12)$, 
defined by convention so that $n+18$ is composite, then occurs 
whenever (\seqnum{A046118}, \seqnum{A046124}) 
\begin{align*} 
\mathsmaller{ 
     P_{\SPTriple,1}(n) (n-1)! + P_{\SPTriple,2}(n) (n-1)!^{2} + 
     P_{\SPTriple,3}(n) (n-1)!^{3} 
} & \equiv -1 \pmod{n(n+6)(n+12)}, 
\end{align*} 
where the three polynomials, $P_{\SPTriple,i}(n)$ for 
$i \defequals 1,2,3$, in the 
previous equation are expanded by the definitions given in the 
following equations: 
\begin{align*} 
P_{\SPTriple,1}(n) & \defequals 
     1 + \Pochhammer{n}{6} + \Pochhammer{n}{12} \\ 
P_{\SPTriple,2}(n) & \defequals 
     \Pochhammer{n}{6} + \Pochhammer{n}{12} + 
     \Pochhammer{n}{6} \times \Pochhammer{n}{12} \\ 
P_{\SPTriple,3}(n) & \defequals 
     \Pochhammer{n}{6} \times \Pochhammer{n}{12}. 
\end{align*} 
Let the congruence parameters in 
\eqref{eqn_FwnXpXtXk_RHS_CongruenceFn_def-stmt_v1} 
corresponding to the sexy prime triplet congruence expansions of the 
single factorial function powers from the previous equations be 
defined as follows: 
\begin{align*} 
\tagtext{Sexy Prime Triplet Congruence Parameters} 
 & \left(\omega, N_{\omega,p}(n), M_{\omega}(n)\right) \\ 
     & \phantom{\qquad} \defmapsto 
     \mathsmaller{ 
     \left(
     \SPTriple, \frac{(-1)^{p}}{72}\left( 
     (n+6)(n+12) - 2 n(n+12) \cdot 7^{p} + n(n+6) \cdot 13^{p} 
     \right), n(n+6)(n+12) 
     \right)
     }. 
\end{align*} 
We similarly see that the elements of an odd integer triple of the 
form $(n, n+6, n+12)$, are all prime whenever $n \geq 3$ satisfies the 
next divisibility requirement modulo the integer triple products, 
$n(n+6)(n+12)$. 
\begin{align*} 
\tagtext{Sexy Prime Triplets} 
\sum\limits_{1 \leq i \leq 3} P_{\SPTriple,i}(n) \times 
     {F_{\SPTriple,n}(1, 1, 1)}^{i} & \equiv -1 
     \pmod{M_{\SPTriple}(n)} 
\end{align*} 
The other notable special case triples of interest in the 
print references, and 
in the additional polynomial congruence cases computed in the 
supplementary reference data \cite{SUMMARYNBREF-STUB}, 
include applications to the 
prime $3$-tuples of the forms $(p+d_1, p+d_2, p+d_3)$ for 
$(d_1, d_2, d_3) \in \left\{(0,2,6), (0,4,6)\right\}$ 
(\cite[\cf \S 1.4]{HARDYWRIGHTNUMT}, \cite[\S 4.4]{PRIMEREC}, 
\seqnum{A022004}, \seqnum{A022005}). 
\end{example} 

\subsubsection{Remarks on expansions of the Stirling number triangles by 
               the $r$-order harmonic numbers} 
\label{subsubSection_remark_SNum_R-OrderHNum_SeqExpIdents_spcases_v1} 

The divisibility of the Stirling numbers of the first kind in 
\eqref{eqn_GenFactFnSeqs_MultipleSum_Idents_exps-stmts_v1} and in 
\eqref{eqn_SingleFactFn_TripleSum_Ident_exps-stmts_v1} 
is tied to well-known expansions of the triangle involving the 
generalized \emph{$r$-order harmonic numbers}, 
$H_n^{(r)} \defequals \sum_{k=1}^{n} k^{-r}$, 
for integer-order $r \geq 1$ \cite[\S 6]{GKP} 
\cite[\cf \S 5.7]{ADVCOMB} \cite[\cf \S 7-8]{HARDYWRIGHTNUMT}. 
The applications cited in the references 
provide statements of the following established 
special case identities for these coefficients 
\cite[\S 4.3]{MULTIFACTJIS} \cite[\S 6.3]{GKP}  
(\seqnum{A001008}, \seqnum{A002805}, \seqnum{A007406}, \seqnum{A007407}, 
\seqnum{A007408}, \seqnum{A007409}): 
\begin{align*} 
\tagtext{Harmonic Number Expansions of the Stirling Numbers} 
\gkpSI{n+1}{2} & = n! \cdot H_n \\ 
\gkpSI{n+1}{3} & = \frac{n!}{2}\left(H_n^2 - H_n^{(2)}\right) \\ 
\gkpSI{n+1}{4} & = 
     \frac{n!}{6}\left(H_n^3 - 3 H_n H_n^{(2)} + 2 H_n^{(3)}\right) \\ 
\tagonce\label{eqn_S1k234_HNum_exp_idents-restmts_v1} 
\gkpSI{n+1}{5} & = 
     \frac{n!}{24}\left( 
     H_n^4 - 6 H_n^{2} H_n^{(2)} + 
     3 \left(H_n^{(2)}\right)^{2} + 8 H_n H_n^{(3)} - 6 H_n^{(4)}\right). 
\end{align*} 
The reference \cite[p.\ 554, Ex.\ 6.51]{GKP} 
gives a related precise statement of the 
necessary condition on the primality of odd integers $p > 3$ 
implied by Wolstenholme's theorem in the following form 
\cite[\cf \S 7.8]{HARDYWRIGHTNUMT}: 
\begin{align*} 
\tagtext{Stirling Number Variant of Wolstenholme's Theorem} 
p > 3 \text{ prime } & \implies \\ 
     & \phantom{\quad} 
     p^2 \mid \mathsmaller{\gkpSI{p}{2}},\ 
     p^2 \mid \mathsmaller{p \gkpSI{p}{3} - p^{2} \gkpSI{p}{4} + 
                           \cdots + p^{p-2} \gkpSI{p}{p}}. 
\end{align*} 
The expansions given in the next remarks of 
Section \ref{subsubSection_MoreGeneralExps_congruences_multiple_factfns} 
suggest similar expansions of congruences involving the 
$\alpha$-factorial functions through more general cases 
$r$-order harmonic number sequences, such as the sequence variants, 
$H_{n,\alpha}^{(r)}$, defined in the next subsection of the article. 

\subsubsection{More general expansions of the new congruence results by 
               multiple factorial functions and 
               generalized harmonic number sequences} 
\label{subsubSection_MoreGeneralExps_congruences_multiple_factfns} 

The noted relations of the divisibility of the 
Stirling numbers of the first kind to the 
$r$-order harmonic number sequences expanded by the special cases from 
\eqref{eqn_S1k234_HNum_exp_idents-restmts_v1} 
are generalized to the $\alpha$-factorial function coefficient 
cases through the following forms of the 
exponential generating functions given in the reference 
\cite{MULTIFACT-CFRACS} \cite[\cf \S 3.3]{MULTIFACTJIS}: 
\begin{align*} 
\tagonce\label{eqn_FcfIIAlphanp1mp1_two-variable_EGFwz} 
\sum_{n \geq 0} \FcfII{\alpha}{n+1}{m+1} \frac{z^n}{n!} & = 
          \frac{(1- \alpha z)^{-1 / \alpha}}{m! \cdot \alpha^{m}} \times 
          \Log\left(\frac{1}{1-\alpha z}\right)^{m}. 
\end{align*} 
The special cases of these coefficients generated by the 
previous equation when $m \defequals 1,2$ are then expanded by the 
sums involving the $r$-order harmonic number sequences in the 
following equations: 
\begin{align*} 
\tagonce\label{eqn_FcfAlphaGenCoeffs_HNumExpIdents-stmts_v1} 
\FcfII{\alpha}{n+1}{2} \frac{1}{n!} & = 
     \alpha^{n-1} \times \sum_{k=0}^{n} 
     \binom{1-\frac{1}{\alpha}}{k} (-1)^{k} H_{n-k} \\ 
\FcfII{\alpha}{n+1}{3} \frac{1}{n!} & = 
     \frac{\alpha^{n-2}}{2} \times \sum_{k=0}^{n} 
     \binom{1-\frac{1}{\alpha}}{k} (-1)^{k} \left( 
     H_{n-k}^2 - H_{n-k}^{(2)} 
     \right). 
\end{align*} 
Identities providing expansions of the 
generalized $\alpha$-factorial triangles from the reference \cite{MULTIFACTJIS} 
at other specific cases of the lower indices $m \geq 3$ 
that involve the slightly generalized cases of the 
ordinary $r$-order harmonic number sequences, $H_{\alpha,n}^{(r)}$, 
defined by the next equation 
are expanded by related constructions. 
\begin{align*} 
\tagtext{Generalized Harmonic Number Definitions} 
H_{\alpha,n}^{(r)} & \defequals 
     \sum_{k=1}^{n} \frac{1}{(\alpha k+1-\alpha)^{r}},\ 
     n \geq 1, \alpha, r > 0 
\end{align*} 
For comparison with the Stirling number identities noted in 
\eqref{eqn_S1k234_HNum_exp_idents-restmts_v1} above, the 
first few cases of the coefficient identities in 
\eqref{eqn_FcfAlphaGenCoeffs_HNumExpIdents-stmts_v1} 
are expanded explicitly by these more general integer-order 
harmonic number sequence cases in the following equations: 
\begin{align*} 
\FcfII{\alpha}{n+1}{2} \frac{1}{n!} & = 
     \alpha^{n} \binom{n+\frac{1-\alpha}{\alpha}}{n} \times 
     H_{n,\alpha}^{(1)} \\ 
\FcfII{\alpha}{n+1}{3} \frac{1}{n!} & = 
     \frac{\alpha^{n}}{2} \binom{n+\frac{1-\alpha}{\alpha}}{n} \times \left( 
     \left(H_{n,\alpha}^{(1)}\right)^{2} - H_{n,\alpha}^{(2)} 
     \right) \\ 
\FcfII{\alpha}{n+1}{4} \frac{1}{n!} & = 
     \frac{\alpha^{n}}{6} \binom{n+\frac{1-\alpha}{\alpha}}{n} \times \left( 
     \left(H_{n,\alpha}^{(1)}\right)^{3} - 
     3 H_{n,\alpha}^{(1)} H_{n,\alpha}^{(2)} + 
     2 H_{n,\alpha}^{(3)} 
     \right). 
\end{align*} 
When $\alpha \defequals 2$, we have a relation between the 
sequences, $H_{n,\alpha}^{(r)}$, and the $r$-order harmonic numbers of the 
form $H_{2,n}^{(r)} = H_{2n}^{(r)} - 2^{-r} H_n^{(r)}$, which yields 
particular coefficient expansions for the double factorial 
functions involved in stating several of the 
congruence results from the examples given below. 
The expansions of the prime-related congruences involving the 
double factorial function cited in 
Section \ref{subsubSection-example_OtherRelatedCongruences_DblFactFns} 
above also suggest additional applications to finding 
integer congruence properties and necessary conditions 
involving these harmonic number sequences related to other more general 
forms of these expansions for prime pairs and prime-related subsequences 
(see Section 
\ref{subsubSection_remark_OtherApps_of_WThm_and_NewCongProps_to_PrimeSubseqs}). 

\subsection{Expansions of several new forms of 
            prime-related congruences and 
            other prime subsequence identities} 
\label{subsubSection_Examples-remarks_RelatedCongruences} 

\subsubsection{Statements of several new 
               results providing finite sum expansions of the 
               single factorial function modulo fixed integers} 

\label{remark_lemma_new_congruences_for_the_SgFactFn} 
The second cases of the generalized factorial function congruences in 
\eqref{eqn_pnAlphaR_seqs_finite_sum_reps_modulop-stmts_v1} 
are of particular utility in expanding several of the non-trivial 
results given in 
Section \ref{subsubSection_NewIdentsFromFiniteDiffEqns_ExpsOfPrime-RelatedCongr} 
below when $h - (n-s) \geq 1$. 
The results related to the double factorial functions and the 
central binomial coefficients expanded through the congruences in 
Section \ref{subsubSection-example_OtherRelatedCongruences_DblFactFns} 
employ the second cases of 
\eqref{eqn_pnAlphaR_seqs_finite_sum_reps_modulop-stmts_v1} and 
\eqref{eqn_pnAlphaR_seqs_finite_sum_reps_modulop-stmts_v2} stated in 
Proposition \ref{prop_ExactFormulas_CongruencesModh_from_FiniteDiffEqns}. 
The next few results stated in 
\eqref{eqn_lemma_Chn11_Chnm1nms_SgFactFnCongruences-exps_v1} and 
\eqref{eqn_lemma_Chn11_Chnm1nms_SgFactFnCongruences-exps_v2} are 
provided as lemmas needed to state many of the congruence results for the 
prime-related sequence cases given as examples in this section and in 
Section \ref{subsubSection-Examples_SomeResults_for_Prime_k-Tuples}. 

\begin{lemma}[Congruences for the Single Factorial Function] 
\label{lemma_lm} 
For natural numbers, $n,n-s \geq 0$, the single factorial function, $(n-s)!$, 
satisfies the following congruences 
whenever $h \geq 2$ is fixed (or when $h$ 
corresponds to some fixed function with an implicit dependence on the 
sequence index $n$): 
\StartGroupingSubEquations 
\label{eqn_lemma_Chn11_Chnm1nms_SgFactFnCongruences-exps_subeqns_ref} 
\begin{align*} 
\tagonce\label{eqn_lemma_Chn11_Chnm1nms_SgFactFnCongruences-exps_v1} 
(n-s)! 
     & \equiv 
     C_{h,n-s}(1, 1) && \pmod{h} \\ 
     & = 
     \sum_{i=0}^{n-s} \binom{h}{i} \Pochhammer{-h}{i} (n-s-i)! && \\ 
     & = 
     \sum_{i=0}^{n-s} \binom{h}{i}^{2} (-1)^{i} i! (n-s-i)! && \\ 
     & = 
     \sum_{i=0}^{n-s} \binom{h}{i} \binom{i-h-1}{i} i! (n-s-i)! && \\ 
\tagonce\label{eqn_lemma_Chn11_Chnm1nms_SgFactFnCongruences-exps_v2}  
(n-s)! 
     & \equiv 
     C_{h,n-s}(-1, n-s) && \pmod{h} \\ 
     & = 
     \sum_{i=0}^{n-s} \binom{h}{i} 
     \Pochhammer{n+1-s-h}{i} \times 
     (-1)^{n-s-i} \Pochhammer{-(n-s)}{n-s-i} && \\ 
     & = 
     \sum_{i=0}^{n-s} \binom{h}{i} \binom{n-s}{i} \binom{h-n+s-1}{i} 
     (-1)^{i} i! \times (n-s-i)!. && 
\end{align*} 
\EndGroupingSubEquations 
The right-hand-side terms, $C_{h,n-s}(\alpha, R)$, in the 
previous two equations correspond to the auxiliary convergent function 
sequences implicitly defined by 
\eqref{eqn_Vandermonde-like_PHSymb_exps_of_PhzCfs}, and the 
corresponding multiple sum expansions stated in 
\eqref{eqn_Chn_formula_stmts}, whose expansions are 
highlighted by the listings given in the tables from the reference 
\cite[\S 9]{MULTIFACT-CFRACS}. 
\end{lemma} 
\begin{proof} 
Since $n! = \pn{n}{-1}{n}$ and $n! = \pn{n}{1}{1}$ for all $n \geq 1$, the 
identities in \eqref{eqn_pnAlphaR_seqs_finite_sum_reps_modulop-stmts_v1} of 
Proposition \ref{prop_ExactFormulas_CongruencesModh_from_FiniteDiffEqns} 
imply the pair of congruences stated in each of 
\eqref{eqn_lemma_Chn11_Chnm1nms_SgFactFnCongruences-exps_v1} and 
\eqref{eqn_lemma_Chn11_Chnm1nms_SgFactFnCongruences-exps_v2} 
modulo any fixed, prescribed setting of the integer-valued $h \geq 2$. 
The expansions of the remaining sums follow first from 
\eqref{eqn_Vandermonde-like_PHSymb_exps_of_PhzCfs}, and then from the 
results stated in Lemma $12$ from the reference \cite{MULTIFACT-CFRACS} 
applied to each of the expansions of these first two sums. 
\end{proof} 

\begin{cor}[Special Cases]  
If $n,n-s,d,an+r \in \mathbb{Z}^{+}$ are selected so that 
$n+d, an+r > n-s$, the coefficient identities for the sequences, 
$\pn{n-s}{1}{1} = [z^{n-s}] \ConvFP{n+d}{1}{1}{z} \pmod{n+d, an+r}$, stated in 
\eqref{eqn_Vandermonde-like_PHSymb_exps_of_PhzCfs-stmt_v1} and 
\eqref{eqn_Chn_formula_stmts} provide that 
\begin{align*} 
\tagonce\label{eqn_SingFactFn_nms_first_ChnSumExps_Modnpd_Modanpr-stmts_v1} 
(n-s)! 
     & \equiv \sum_{i=0}^{n-s} \binom{n+d}{i}^{2} (-1)^{i} i! \times 
     (n-s-i)! && \pmod{n+d} \\ 
(n-s)! 
     & \equiv \sum_{i=0}^{n-s} \binom{an+r}{i} 
     \Pochhammer{-(an+r)}{i} (n-s-i)! && \pmod{an+r}. 
\end{align*} 
\end{cor} 
\begin{proof} 
The expansions of the congruences for the single factorial function 
provided by the lemmas stated in 
\eqref{eqn_lemma_Chn11_Chnm1nms_SgFactFnCongruences-exps_v1} 
follow as immediate consequences of the results in 
Proposition \ref{prop_ExactFormulas_CongruencesModh_from_FiniteDiffEqns}. 
The previous equations then correspond to the particular cases of these 
results when $h \defmapsto n+d$ and $h \defmapsto an+r$ 
respective order of the equations stated above. 
\end{proof} 
The results expanded through the symbolic computations with these sums 
obtained from \Mm{}'s \SigmaPkg package outlined in 
Section \ref{subsubSection_Examples-remarks_RelatedCongruences-MmCompsWith_the_SigmaPkg}
provide additional non-trivial forms of the prime-related congruences 
involving the single factorial function cases defined by the previous 
two results. 

\subsubsection{Examples: Consequences of Wilson's theorem} 
\label{subsubSection_Examples_ConsequencesOfWThm} 

\sublabel{Expansions of variants of Wilson's theorem} 
The previous identities lead to additional examples phrasing 
congruences equivalent to the primality condition in 
Wilson's theorem involving products of the 
single factorial functions, $n!$ and $(n+1)!$, 
modulo some odd integer $p \defequals 2n+1$ of unspecified primality to be 
determined by an application of these results. 
For example, we can prove that for $n \geq 1$, 
an odd integer $p \defequals 2n+1$ is prime if and only if
\cite[\cf \S 8.9]{HARDYWRIGHTNUMT}\footnote{ 
     The first equation restates a result proved by Sz\'{a}nt\'{o}  
     in 2005 given on the 
     \href{http://mathworld.wolfram.com/WilsonsTheorem.html}{
     \emph{MathWorld}} website. 
} 
\begin{align*} 
\tagonce\label{eqn_MathWorld_FormsOf_WilsonsThm} 
2^{1-n} \cdot n! \cdot (n+1)! & \equiv (-1)^{\binom{n+2}{2}} && \pmod{2n+1} \\ 
\left( n! \right)^{2} & \equiv (-1)^{n+1} && \pmod{2n+1}. 
\end{align*} 
The first congruence in 
\eqref{eqn_MathWorld_FormsOf_WilsonsThm} 
yields the following additional forms of 
necessary and sufficient conditions on the primality of the odd integers, 
$p \defequals 2n+1$, resulting from Wilson's theorem: 
\begin{align*} 
\tagonce\label{eqn_MathWorld_FormsOf_WilsonsThm-Chm_prodsum_exp-stmt_v2} 
\frac{1}{2^{n-1}} \times \left( 
     \prod\limits_{s \in \{0,1\}} \sum_{i=0}^{n+s} 
     \binom{2n+1}{i}^2 (-1)^{i} i! (n+s-i)! 
     \right) & \equiv (-1)^{(n+1)(n+2) / 2} && \pmod{2n+1} \\ 
\left(\sum_{i=0}^{n} \binom{2n+1}{i}^{2} (-1)^{i} i! (n-i)!\right)^{2} 
     & \equiv (-1)^{n+1} && \pmod{2n+1}. 
\end{align*} 
\sublabel{Congruences for primes of the form $n^2+1$} 
If we further seek to determine new properties of the odd 
primes of the form $p \defequals n^2 + 1 \geq 5$, 
obtained from adaptations of the new forms given by these sums, 
the second consequence of Wilson's theorem provided in 
\eqref{eqn_MathWorld_FormsOf_WilsonsThm} above 
leads to an analogous requirement expanded in the form of the 
next equations \cite[\S 3.4(D)]{PRIMEREC} (\seqnum{A002496}). 
\begin{align*} 
n^2 + 1 \text{ prime } & \iff 
     \mathsmaller{ 
     \left(\sum_{i=0}^{n^2 / 2} 
     \binom{n^2+1}{i} \Pochhammer{-(n^2+1)}{i} 
     \left(\mathsmaller{\frac{1}{2} (n^2 - 2i)}\right)! 
     \right)^{2} 
     } 
     && \equiv (-1)^{n^2 / 2 + 1} && \pmod{n^2+1} \\ 
\phantom{n^2 + 1 \text{ prime }} & \iff 
     \mathsmaller{ 
     \left(\sum_{i=0}^{n^2 / 2} 
     \binom{n^2+1}{i} \binom{i-n^2-2}{i} i! 
     \left(\mathsmaller{\frac{1}{2} (n^2 - 2i)}\right)! 
     \right)^{2} 
     } 
     && \equiv (-1)^{n^2 / 2 + 1} && \pmod{n^2+1} 
\end{align*} 
For comparison with the previous two congruences, the 
first classical statement of Wilson's theorem 
stated as in the introduction is paired 
with the next expansions of the fourth and fifth multiple sums stated in 
\eqref{eqn_Chn_formula_stmts} 
to show that an odd integer $p \geq 5$ of the form 
$p \defequals n^2+1$ is prime for some even $n \geq 2$ whenever 
\begin{align*} 
\sum\limits_{\substack{0 \leq m \leq k \leq n^2 \\ 
             0 \leq v \leq i \leq s \leq n^2} 
             } & 
     \underset{C_{h,k}(\alpha, R) \text{ where } 
          h \defmapsto n^2+1,\ k \defmapsto n^2,\ 
          \alpha \defmapsto -1,\ R \defmapsto n^2 
          \text{ in \eqref{eqn_Chn_formula_stmts-exp_v4} }}{\underline{ 
     \mathsmaller{
     \binom{n^2+1}{k} \binom{m}{s} \binom{i}{v} \binom{n^2+1+v}{v} 
     \gkpSI{k}{m} \gkpSII{s}{i} (-1)^{m+i-v} i! \Pochhammer{-n^2}{n^2-k} 
     \left(n^2 + 1\right)^{m-s}}}} 
     \equiv -1 \hspace{-3.5mm} \pmod{n^2+1} \\ 
\sum\limits_{\substack{0 \leq i \leq n^2 \\ 
                       0 \leq m \leq k \leq n^2 \\ 
                       0 \leq t \leq s \leq n^2} 
                 } & 
     \underset{C_{h,n}(\alpha, R) \text{ where } 
          h \defmapsto n^2+1,\  n \defmapsto n^2,\ 
          \alpha \defmapsto 1,\ R \defmapsto 1 
          \text{ in \eqref{eqn_Chn_formula_stmts-exp_v5} }}{\underline{ 
     \mathsmaller{
     \binom{n^2+1}{k} \binom{m}{t} \gkpSI{k}{m} 
     \gkpSI{n^2-k}{s-t} \gkpSII{s}{i} 
     (-1)^{m+s-i} n^{2m-2t} 
     \times i! 
     }} 
     } 
     \equiv -1 \pmod{n^2+1}. 
\end{align*} 
These congruences are straightforward to adapt to 
form related results characterizing subsequences of primes of the form 
$p \defequals an^2+bn+c$ for some fixed constants 
$a,b,c \in \mathbb{Z}$ satisfying the constraints given in the 
reference at natural numbers $n \geq 1$ 
\cite[\S 2.8]{HARDYWRIGHTNUMT}. 

\sublabel{Congruences for the Wilson primes} 
The sequence of \emph{Wilson primes} denotes the subsequence odd primes $n$ 
such that the \emph{Wilson quotient}, 
$\WilsonQuotient{n} \defequals 
 \frac{\left((n-1)! + 1\right)}{n}$, is divisible by $n$, 
or equivalently the sequence of odd integers $n \geq 3$ 
with the divisibility property of the single factorial 
function, $(n-1)!$, 
modulo $n^2$ defined in the next equation 
\cite[\S 5.4]{PRIMEREC} \cite[\S 6.6]{HARDYWRIGHTNUMT} 
(\seqnum{A007619}, \seqnum{A007540}). 
\begin{align*} 
\tagtext{Wilson Primes} 
\WilsonPrimeSet & \defequals \left\{ 
     n \geq 3 : n^2 | (n-1)! + 1 
     \right\} 
     \quad \seqmapsto{A007540} \left(5, 13, 567, \ldots \right). 
\end{align*} 
A few additional expansions of congruences characterizing the 
Wilson primes correspond to the imposing the following additional 
equivalent requirements on the divisibility of the 
single factorial function (modulo $n$) in Wilson's theorem: 
\begin{align*} 
\tagtext{Wilson Prime Congruences} 
\underset{C_{n^2,n-1}(1, 1)\ \equiv\ (n-1)! \pmod{n^2}}{
     \underbrace{ 
     \sum_{i=0}^{n-1} \binom{n^2}{i}^{2} (-1)^{i} i! (n-1-i)! 
     } 
} 
     & \equiv -1 && \pmod{n^2} \\ 
\underset{(n-1)! \pmod{n^2}}{
     \underbrace{ 
     \sum_{i=0}^{n-1} \binom{n^2}{i} \binom{i-n^2-1}{i} i! (n-1-i)! 
     } 
} 
     & \equiv -1 && \pmod{n^2} \\ 
\underset{C_{n^2,n-1}(-1, n-1)\ \equiv\ (n-1)! \pmod{n^2}}{
     \underbrace{ 
     \sum_{i=0}^{n-1} \binom{n^2}{i} 
     \FFactII{(n^2-n)}{i} \times (-1)^{n-1-i} \FFactII{(n-1)}{n-1-i} 
     } 
} 
     & \equiv -1 && \pmod{n^2}. 
\end{align*} 
The congruences in the previous equation 
are verified numerically in the reference \cite{SUMMARYNBREF-STUB} to 
hold for the first few hundred primes, $p_n$, 
only when $p_n \in \{5, 13, 563\}$. 
The third and fourth multiple sum expansions of the coefficients, 
$C_{n^2,n-1}(-1, n-1)$, given in 
\eqref{eqn_Chn_formula_stmts} 
similarly provide that an odd integer $n > 3$ is a 
Wilson prime if and only if either of the following pair of 
congruences holds modulo the integer squares $n^2$: 
\begin{align*} 
\undersetbrace{C_{n^2,n-1}(-1, n-1) 
     \text{ in \eqref{eqn_Chn_formula_stmts-exp_v3} }}{ 
     \mathsmaller{ 
     \sum\limits_{s=0}^{n-1} \sum\limits_{i=0}^{s} \left( 
     \sum\limits_{k=0}^{n-1} \sum\limits_{m=0}^{k} 
     \binom{n^2}{k} \binom{n^2}{i} \binom{m}{s} 
     \gkpSI{k}{m} \gkpSII{s}{i} (-1)^{n-1-k} 
     \Pochhammer{1-n}{n-1-k} (-n)^{m-s} i! 
     \right)} 
     } 
     & 
     \mathsmaller{ 
     \equiv -1 \pmod{n^2} 
     } \\ 
     \undersetbrace{C_{n^2,n-1}(-1, n-1) 
     \text{ in \eqref{eqn_Chn_formula_stmts-exp_v4} }}{ 
\mathsmaller{ 
     \sum\limits_{s=0}^{n-1} \sum\limits_{i=0}^{s} \sum\limits_{v=0}^{i} 
     \left( 
     \sum\limits_{k=0}^{n-1} \sum\limits_{m=0}^{k} 
     \binom{n^2}{k} \binom{m}{s} \binom{i}{v} \binom{n^2+v}{v} 
     \gkpSI{k}{m} \gkpSII{s}{i} (-1)^{i-v} 
     \FFactII{(n-1)}{n-1-k} (-n)^{m-s} i! 
     \right)} 
     } 
     & 
     \mathsmaller{ 
     \equiv -1 \pmod{n^2}. 
     } 
\end{align*} 

\sublabel{Congruences for special prime pair sequences} 
The constructions of the new results expanded above 
are combined with the known congruences established in the reference 
\cite[\S 3, \S 5]{ONWTHM-AND-POLIGNAC-CONJ} to obtain the 
alternate necessary and sufficient conditions for the 
twin prime pairs 
stated in \eqref{eqn_TwinPrime_NewExpsOfKnownCongruenceResults-stmts_v1} 
of the introduction (\seqnum{A001359}, \seqnum{A001097}). 
The results in the references also provide analogous 
expansions of the congruence statements 
corresponding to characterizations of the 
\emph{cousin prime} and \emph{sexy prime} pairs 
expanded in the following equations (\seqnum{A023200}, \seqnum{A023201}): 
\begin{align*} 
\tagtext{Cousin Prime Pairs} 
 & 2n+1, 2n+5 \text{ odd primes } && \\ 
 & \qquad \iff 
   \mathsmaller{ 
     36\left(\sum\limits_{i=0}^{n} 
     \binom{(2n+1)(2n+5)}{i}^2 (-1)^i i! (n-i)!\right)^2} & \\ 
 & \phantom{\qquad\iff} + 
   (-1)^{n} (29-14n) && \equiv 0 \pmod{(2n+1)(2n+5)} \\ 
 & \qquad \iff 
   \mathsmaller{
     96 \left(\sum\limits_{i=0}^{2n} 
     \binom{(2n+1)(2n+5)}{i}^2 (-1)^i i! (2n-i)!\right)} & \\ 
 & \phantom{\qquad\iff} + 
   46n+119 && \equiv 0 \pmod{(2n+1)(2n+5)} \\ 
 & 2n+1, 2n+7 \text{ odd primes } && \\ 
\tagtext{Sexy Prime Pairs} 
 & \qquad \iff 
   \mathsmaller{ 
     1350\left(\sum\limits_{i=0}^{n} 
     \binom{(2n+1)(2n+7)}{i}^2 (-1)^i i! (n-i)!\right)^2} & \\ 
 & \phantom{\qquad\iff} + 
   (-1)^{n} (578n+1639) && \equiv 0 \pmod{(2n+1)(2n+7)} \\ 
 & \qquad \iff 
   \mathsmaller{
     4320\left(\sum\limits_{i=0}^{2n} 
     \binom{(2n+1)(2n+7)}{i}^2 (-1)^i i! (2n-i)!\right)} & \\ 
\tagonce\label{eqn_CousinSexyPrimePairs_CongruenceStmts-exps_v1} 
 & \phantom{\qquad\iff} + 
   1438n+5039 && \equiv 0 \pmod{(2n+1)(2n+7)}. 
\end{align*} 
The multiple sum expansions of the single factorial functions in the 
congruences given in the previous two examples also yield 
similar restatements of the pair of congruences in 
\eqref{eqn_TwinPrime_NewExpsOfKnownCongruenceResults-stmts_v1} 
from Section \ref{subSection_Wthm_CThm_SpCase_Apps} 
providing that for some $n \geq 1$, the odd integers, 
$(p_1, p_2) \defequals (2n+1, 2n+3)$, are both prime 
whenever either of the following divisibility conditions hold: 
\begin{align*} 
2 & \times 
     \undersetbrace{C_{(2n+1)(2n+3),n}(-1,n) 
     \text{ in \eqref{eqn_Chn_formula_stmts-exp_v3} }}{\left(
     \mathsmaller{ 
     \sum\limits_{\substack{0 \leq i \leq s \leq n \\ 
                       0 \leq m \leq k \leq n}} 
     \binom{(2n+1)(2n+3)}{i} \binom{(2n+1)(2n+3)}{k} \binom{m}{s} 
     \gkpSI{k}{m} \gkpSII{s}{i} (-1)^{s+k} i! \times 
     \FFactII{n}{n-k} (n+1)^{m-s} 
     } 
     \right)}^{2}  \\ 
     & \phantom{\quad} + 
     (-1)^{n} (10n+7) \equiv 0 \hspace{2.6in} \pmod{(2n+1)(2n+3)} \\ 
4 & 
     \undersetbrace{C_{(2n+1)(2n+3),2n}(-1,2n) 
     \text{ in \eqref{eqn_Chn_formula_stmts-exp_v4} }}{ 
     \left( 
     \mathsmaller{ 
     \sum\limits_{\substack{0 \leq v \leq i \leq s \leq 2n \\ 
                         0 \leq m \leq k \leq 2n}} 
     \binom{(2n+1)(2n+3)}{k} \binom{(2n+1)(2n+3)+v}{v} 
     \binom{i}{v} \binom{m}{s} \gkpSI{k}{m} \gkpSII{s}{i} 
     (-1)^{s-i+v+k} i! \times 
     \FFactII{(2n)}{2n-k} (2n+1)^{m-s} 
     } 
     \right)}  \\ 
     & \phantom{\quad} + 
     (2n+5) \equiv 0 \hspace{3.1in} \pmod{(2n+1)(2n+3)}. 
\end{align*} 
The treatment of the integer congruence identities involved in 
these few notable example cases from 
Example \ref{example_TripleSumIdents_App_to_WThm}, in 
Example \ref{example_FirstSPT_Result}, and in the 
last several examples from the remarks above, 
is by no means exhaustive, but serves to demonstrate the utility 
of this approach in formulating several new forms of 
non-trivial prime number results 
with many notable applications. 

\subsubsection{Computations of symbolic sums with Mathematica's Sigma package} 
\label{subsubSection_Examples-remarks_RelatedCongruences-MmCompsWith_the_SigmaPkg}

\begin{example}[Expansions of the First Sum from Lemma \ref{lemma_lm}] 
The working summary notebook 
attached to the article \cite{SUMMARYNBREF-STUB} 
includes computations with \Mm's \SigmaPkg package 
that yield additional forms of the identities expanded in 
\eqref{eqn_SingFactFn_nms_first_ChnSumExps_Modnpd_Modanpr-stmts_v1}, 
\eqref{eqn_MathWorld_FormsOf_WilsonsThm}, and 
\eqref{eqn_MathWorld_FormsOf_WilsonsThm-Chm_prodsum_exp-stmt_v2}, 
for the single factorial function, $(n-s)!$, when $s \defequals 0$. 
For example, 
alternate variants of the identity for the first sum in 
\eqref{eqn_SingFactFn_nms_first_ChnSumExps_Modnpd_Modanpr-stmts_v1} are 
expanded as follows: 
\begin{align*} 
n! & \equiv 
     \sum_{i=0}^{n} \binom{n+d}{i} 
     \Pochhammer{-(n+d)}{i} (n-i)! && \pmod{n+d} \\ 
n! & \equiv 
     \sum_{i=0}^{n} \binom{n+d}{i} \binom{i-n-d-1}{i} i! (n-i)! 
     && \pmod{n+d} \\ 
n! & \equiv 
     \sum_{i=0}^{n} \binom{n+d}{i}^{2} (-1)^{i} i! (n-i)! && \pmod{n+d} \\ 
   & = 
     (-1)^{n} \Pochhammer{2d}{n} \times \left( 
     1 + d^2 \times \sum_{i=1}^{n} \binom{i+d}{i} 
     \frac{(-1)^{i} \Pochhammer{-(i+d)}{i}}{(i+d)^2 \Pochhammer{2d}{i}} 
     \right) && \\ 
\tagonce\label{eqn_SingFactFn_nms_first_ChnSumExps_Modnpd_Modanpr-stmts_v2} 
   & = 
   \undersetbrace{\defequals S_{1,d}(n)}{ 
   (-1)^{n} d^{2} \times \sum_{i=0}^{n} \binom{i+d}{d}^{2} \times 
   \frac{i! \cdot \Pochhammer{2d+i}{n-i}}{(i+d)^{2}}
   }. && 
\end{align*} 
The first special cases of the sums, $S_{1,d}(n)$, 
defined in the last equation 
are expanded in terms of the first-order harmonic numbers 
for integer-valued cases of $d \geq 1$ as follows: 
\begin{align*} 
S_{1,1}(n) & = 
     (-1)^{n} \Pochhammer{2}{n} \times H_{n+1} \\ 
     & = 
     (-1)^{n} (n+1)! \times H_{n+1} \\ 
S_{1,2}(n) & = 
     (-1)^{n} (n+2)! \times \left( 
     (n+3) H_{n+2} - 2(n+2) 
     \right) \\ 
     & = 
     \frac{3}{2} \times (-1)^{n} \Pochhammer{4}{n} \times \left( 
     H_{n} - \frac{(2n^3+8n^2+7n-1)}{(n+1)(n+2)(n+3)} 
     \right) \\ 
S_{1,3}(n) & = 
     \frac{1}{4} \times (-1)^{n} (n+4)! \times \left( 
     (n+5) H_{n+3} - 3(n+3) 
     \right) \\ 
     & = 
     \frac{10}{3} \times (-1)^{n} \Pochhammer{6}{n} \times \left( 
     H_{n} - \frac{(3n^4+24n^3+60n^2+46n-1)}{(n+1)(n+2)(n+3)(n+5)} 
     \right) \\ 
S_{1,4}(n) & = 
     \frac{1}{108} \times (-1)^{n} (n+4)! \times \left( 
     3(n+5)(n+6)(n+7) H_{n+4} - (n+4)(11n^2+118n+327)  
     \right) \\ 
     & = 
     \frac{35}{12} \times (-1)^{n} \Pochhammer{8}{n} \times \left( 
     \mathsmaller{ 
     3 H_{n} - 
     \frac{(11n^7+260n^6+2498n^5+12404n^4+33329n^3+45548n^2+24426n-108)}{ 
     (n+1)(n+2)(n+3)(n+4)(n+5)(n+6)(n+7)} 
     } 
     \right). 
\end{align*} 
More generally, we can provide a somewhat intricate proof omitted here of 
another expansion of the last sum in 
\eqref{eqn_SingFactFn_nms_first_ChnSumExps_Modnpd_Modanpr-stmts_v2} 
in terms of the first-order harmonic numbers given by 
(\seqnum{A001008}, \seqnum{A002805}) 
\begin{align*} 
n! & \equiv 
     (-1)^{n} \Pochhammer{2d}{n}\left( 1 + 
     \frac{d}{2} \binom{2d}{d} \sum_{1 \leq i \leq d} 
     \frac{(-1)^{d-i} (d+i-2)!}{(i-1)!^2 (d-i)!} \left[ 
     H_{n-1+d+i} - H_{d+i-1} 
     \right]\right) \hspace{-4mm} \pmod{n+d}. 
\end{align*} 
When the parameter $d \defmapsto d_n$ 
in the previous expansions of the sums, $S_{1,d}(n)$, 
depends linearly, or quadratically on $n$, the harmonic-number-based 
identities expanding these congruence forms yield the forms of the 
next examples considered in this subsection. 
\end{example} 

\begin{example}[Expansions of Sums with a Linear Dependence of $h$ on $n$] 
Further computation with \Mm's \SigmaPkg package 
similarly yields the following alternate form of the second sums in 
\eqref{eqn_SingFactFn_nms_first_ChnSumExps_Modnpd_Modanpr-stmts_v1} 
implicit to the congruence identities stated in 
\eqref{eqn_MathWorld_FormsOf_WilsonsThm} and 
\eqref{eqn_MathWorld_FormsOf_WilsonsThm-Chm_prodsum_exp-stmt_v2} above: 
\StartGroupingSubEquations 
\label{eqn_nFactMod2np1_CongruenceIdent_SigmaPkgAltSums} 
\begin{align} 
n! & \equiv 
     \mathsmaller{ 
     \sum\limits_{i=0}^{n} \binom{2n+1}{i}^2 (-1)^{i} 
     i! (n-i)! 
     } 
     \qquad \pmod{2n+1} \\ 
     & = 
     \mathsmaller{ 
     \frac{(-1)^{n} (3n+1)!}{8 \left(n!\right)^{2}} \times \left( 
     8 - \sum\limits_{i=1}^{n} \binom{2i+1}{i}^{2} 
     \frac{(i!)^{3}}{2 \cdot (3i+1)!} \left( 
     11 + \frac{20}{i} - \frac{8}{(2i+1)} + \frac{1}{(2i+1)^2} 
     \right) 
     \right) 
     } \\ 
\notag 
     & = 
     \mathsmaller{ 
     \frac{(-1)^{n} (3n+1)!}{\left(n!\right)^{2}} - 
     \frac{(-1)^{n}}{16} \times 
     \sum\limits_{i=1}^{n} \binom{2i+1}{i}^{2} 
     \frac{i! \cdot \Pochhammer{3i+2}{3n-3i}}{ 
     \Pochhammer{i+1}{n-i}^{2}} \left( 
     11 + \frac{20}{i} - \frac{8}{(2i+1)} + \frac{1}{(2i+1)^2} 
     \right) 
     } \\ 
     & = 
     \mathsmaller{ 
     \frac{(-1)^{n} (3n+1)!}{8 \left(n!\right)^{2}} \times \left( 
     8 - \sum\limits_{i=1}^{n} \binom{2i+1}{i}^{2} 
     \frac{(i!)^{3}}{(3i)!} \left( 
     \frac{10}{i} + \frac{5}{(2i+1)} - \frac{1}{(2i+1)^2} - 
     \frac{32}{(3i+1)} 
     \right) 
     \right) 
     } \\ 
\notag 
     & = 
     \mathsmaller{ 
     \frac{(-1)^{n} (3n+1)!}{\left(n!\right)^{2}} - 
     \frac{(3n+1) (-1)^{n}}{8} \times 
     \sum\limits_{i=1}^{n} \binom{2i+1}{i}^{2} 
     \frac{i! \cdot \Pochhammer{3i+1}{3n-3i}}{ 
     \Pochhammer{i+1}{n-i}^{2}} \left( 
     \frac{10}{i} + \frac{5}{(2i+1)} - \frac{1}{(2i+1)^2} - 
     \frac{32}{(3i+1)} 
     \right) 
     }. 
\end{align} 
\EndGroupingSubEquations 
The documentation for the \SigmaPkg package in the 
reference \cite[Ex.\ 3.3]{SYMB-SUM-COMB-SIGMAPKGDOCS} 
contains several identities 
related to the partial harmonic-number-related 
expansions of the single factorial function sums 
given in the next examples, and for the computations 
contained in the reference \cite{SUMMARYNBREF-STUB}. 
\end{example} 

\begin{example}[Expansions of Sums Involving a Quadratic Dependence of 
                $h$ on $n$] 
The second (non-square) sums implicit to the congruences providing 
characterizations of the twin prime pairs given in 
\eqref{eqn_TwinPrime_NewExpsOfKnownCongruenceResults-stmts_v1} 
of the introduction, and of the 
cousin and sexy primes expanded in 
\eqref{eqn_CousinSexyPrimePairs_CongruenceStmts-exps_v1} 
of the previous subsection, are easily generalized to form related results for 
other prime pairs (\seqnum{A023202}, \seqnum{A023203}, \seqnum{A046133}). 
In particular, 
for positive integers $d \geq 1$, the special cases of these 
expansions for the prime pair sequences considered above lead to 
more general congruence-based characterizations of the odd prime pairs, 
$(2n+1, 2n+1+2d)$, in the form of the following equation 
for some $a_d, b_d, c_d \in \mathbb{Z}$ and where the parameter 
$h_d \defequals (2n+1)(2n+1+2d) > 2n$ implicit to these sums 
depends quadratically on $n$ 
\cite[\cf \S 3, \S 5]{ONWTHM-AND-POLIGNAC-CONJ}: 
\begin{align*} 
\tagonce\label{eqn_GenPrimePairCongFn_hndabc_v1} 
     & 2n+1, 2n+1+2d \text{ prime } \\ 
     & \qquad \iff \quad 
     a_d \times 
     \underset{\mathlarger{ \defequals S_n(h_d) \equiv (2n)! 
     \pmod{h_d}}}{\underbrace{
     \sum_{i=0}^{2n} \binom{h_d}{i}^{2} (-1)^{i} i! (2n-i)!} 
     } + 
     b_d n + c_d \equiv 0 \pmod{h_d}. 
\end{align*} 
For natural numbers $h, i, n \geq 0$, 
let the shorthand for the functions, 
$T_{h,n}$ and $H_{h,i}$, be defined as in the next equations. 
\begin{align*} 
\tagonce\label{eqn_ShorthandFnNotation_Thk_Hhi-defs_exps_v1} 
S_n(h) & \defequals 
     \sum_{i=0}^{2n} \binom{h}{i}^{2} (-1)^{i} i! (2n-i)! \\ 
T_{h,n} & \defequals 
     \prod_{j=1}^{n} \left( 
     \frac{(h-2j)^2 (h+1-2j)^2}{2 \cdot (2h+1-2j) (h-j)} 
     \right) \\ 
     & \phantom{:} = 
     4^{n} \times 
     \frac{\Pochhammer{\frac{1-h}{2}}{n}^2 \Pochhammer{1-\frac{h}{2}}{n}^2}{ 
     \Pochhammer{\frac{1}{2}-h}{n} \Pochhammer{1-h}{n}} = 
     \frac{\Pochhammer{1-h}{2n}^2}{\Pochhammer{1-2h}{2n}} \\ 
H_{h,i} & \defequals 
     \frac{h (h+1) (2h-1)}{(h-1)(h-i)} + 
     \frac{2 (h+1)^2 (2h+1)}{h (2h+1-2i)} + 
     \frac{2 (h+1)}{h (h-1) (h+1-2i)} \\ 
     & \phantom{:} = 
     \frac{(h+1) (2h+1-4i) (h-2i)}{(2h+1-2i) (h+1-2i) (h-i)} 
\end{align*} 
Computations with the \SigmaPkg package 
yield the next alternate expansion of the first sum 
defined in \eqref{eqn_ShorthandFnNotation_Thk_Hhi-defs_exps_v1} 
given by 
\begin{align*} 
S_n(h) & = \sum_{i=0}^{n} \binom{h}{2i}^{2} (2i)! \times 
     \frac{T_{h,n}}{T_{h,i}} \times H_{h,i}, 
\end{align*} 
where the ratios of the product functions in the previous equation 
are simplified by the identities given in the reference 
\cite{WOLFRAMFNSSITE-INTRO-FACTBINOMS} as follows: 
\begin{align*} 
\tagonce\label{eqn_ThiThnFnRatioTerms_PHSymbol_SimplificationIdent} 
\frac{T_{h,n}}{T_{h,i}} & = 
     \frac{\Pochhammer{1-h}{2n}^2}{\Pochhammer{1-2h}{2n}} \times 
     \frac{\Pochhammer{1-2h}{2i}}{\Pochhammer{1-h}{2i}^2} = 
     \frac{\Pochhammer{1-h+2i}{2n-2i}^2}{\Pochhammer{1-2h+2i}{2n-2i}},\ 
     n \geq i. 
\end{align*} 
The forms of the generalized sums, $S_n(h)$, obtained 
from the special case identity above 
using the \SigmaPkg software package routines 
are then expanded by the harmonic-number-related sums over the 
originally fixed indeterminate parameter $h$ in the following forms: 
\begin{align*} 
S_n(h) & = 
     \mathsmaller{ 
     \sum\limits_{i=0}^{n} \binom{h}{2i}^{2} 
     \frac{(2i)! \times 
     \Pochhammer{1-h+2i}{2n-2i}^2}{\Pochhammer{1-2h+2i}{2n-2i}} 
     \times \left( 
     \frac{h (h+1) (2h-1)}{(h-1)(h-i)} + 
     \frac{2 (h+1)^2 (2h+1)}{h (2h+1-2i)} + 
     \frac{2 (h+1)}{h (h-1) (h+1-2i)} 
     \right) 
     } \\ 
     & = 
     \mathsmaller{ 
     \sum\limits_{i=0}^{n} \binom{h}{2n-2i}^{2} 
     (2n-2i)! \times 
     \frac{\Pochhammer{1-h+2n-2i}{2i}^2}{\Pochhammer{1-2h+2n-2i}{2i}} 
     \times \left( 
     \frac{(h+1) (2h+1-4(n-i)) (h-2(n-i))}{(2h+1-2(n-i)) (h+1-2(n-i)) (h-n+i)}
     \right) 
     } 
\end{align*} 
The first sum on the right-hand-side of 
\eqref{eqn_ShorthandFnNotation_Thk_Hhi-defs_exps_v1} 
denotes the special prime pair congruence expansions for the 
twin, cousin, and sexy prime pairs already defined by the 
examples cited in the last sections corresponding to the 
respective forms of \eqref{eqn_GenPrimePairCongFn_hndabc_v1} where 
\[\left(d, a_d, b_d, c_d\right)_{d=1}^{3} \defequals \left\{
 (1, 4, 2, 5), (2, 96, 48, 119), 
 (3, 4320, 1438, 5039) \right\}, 
\] 
and where $h \defmapsto (2n+1)(2n+1+2d)$, 
as computed in the reference \cite{SUMMARYNBREF-STUB}. 
\end{example} 

\begin{remark}
\label{footnote_AltBinomCoeffExps_in_SigmaPrimePairSums} 
\label{subsubSection_FutureResTopics_Rmks_in_SummaryNB} 
Note that the binomial coefficient identity, 
$\binom{\binom{k}{2}}{2} = 3 \binom{k+1}{4}$, 
given in the exercises section of the reference 
\cite[p.\ 535, Ex.\ 5.67]{GKP}, suggests simplifications, or 
``\emph{reductions}'' in order, of the sums, $S_n(h)$, 
when $h$ denotes some fixed, implicit application-dependent 
quadratic function of $n$ obtained by 
first expanding the inner terms, $\binom{h}{i}$, 
as a (finite) linear combination of binomial coefficient terms whose 
upper index corresponds to a linear function of $n$ 
\cite{SUMMARYNBREF-STUB}. 
The \SigmaPkg package is able to obtain alternate forms of these 
pre-processed finite sums defining the functions, 
$S_n(\beta n+\gamma)$, for scalar-valued $\beta,\gamma$ 
that generalize the last two expansions provided above in 
\eqref{eqn_nFactMod2np1_CongruenceIdent_SigmaPkgAltSums}. 
The summary notebook document prepared with this 
manuscript contains additional remarks and examples related to the 
results in the article. 
For example, several specific expansions of the 
binomial coefficients, $\binom{(2n+1)(2n+2d+1)}{i}$, at 
upper index inputs varying quadratically on $n$ suggested by the first 
upper index reduction identity above 
are computed as a starting point for simplifying the terms in these 
sums in the reference \cite{SUMMARYNBREF-STUB}. 
Additional notes providing documentation and more detailed 
computational examples will be added to 
updated versions of the summary reference. 
\end{remark} 

\subsection{Expansions of congruences involving the double factorial function} 
\label{subsubSection-example_OtherRelatedCongruences_DblFactFns} 

\subsubsection{Statements of congruences for the double factorial function} 

\begin{prop}[Congruences for the Double Factorial Function] 
Let $h \geq 2$ be odd or prime and suppose that $s$ is an integer satisfying 
$0 \leq s \leq h$. We have the following congruences for the double 
factorial function, $(2n-1)!!$: 
\begin{align*} 
(2n-1)!! & \equiv 
     \sum_{i=0}^{n} 
     \binom{h}{i} 2^{n+(s+1) i} \Pochhammer{\mathsmaller{\frac{1}{2}}-h}{i} 
     \Pochhammer{\mathsmaller{\frac{1}{2}}}{n-i} 
     && \pmod{2^{s} h} \\ 
     & \equiv 
     \sum_{i=0}^{n} 
     \binom{h}{i} \binom{2n-2i}{n-i} 
     \frac{2^{n+ (s+1) i}}{4^{n-i}} \times 
     \Pochhammer{\mathsmaller{\frac{1}{2}}-h}{i} (n-i)!
     && \pmod{2^{s} h} \\ 
     & \equiv 
     \sum_{i=0}^{n} \binom{h}{i} (-2)^{n+(s+1) i} 
     \Pochhammer{\mathsmaller{\frac{1}{2}}+n-h}{i} 
     \Pochhammer{\mathsmaller{\frac{1}{2}}-n}{n-i} 
     && \pmod{2^{s} h}. 
\end{align*} 
\end{prop} 
\begin{proof} 
The coefficient expansion given by the last identity in 
\eqref{eqn_pnAlphaR_seqs_finite_sum_reps_modulop-stmts_v2} 
provides the alternate forms of congruences for the 
double factorial functions, 
$(2n-1)!! = \pn{n}{-2}{2n-1}$ and $2^{n} \Pochhammer{1/2}{n} = \pn{n}{2}{1}$, 
modulo $2^{s} \cdot h$ stated in the first and third of the previous equations 
for fixed integers $h \geq 2$ and any integer $0 \leq s \leq h$. 
For natural numbers $n \geq 0$, the central binomial coefficients 
satisfy an expansion by the following identity 
given in the reference \cite[\S 5.3]{GKP}: 
\begin{align*} 
\tagtext{Binomial Coefficient Half-Index Identities} 
\Pochhammer{\mathsmaller{\frac{1}{2}}}{n} & = 
     \binom{-\mathsmaller{\frac{1}{2}}}{n} \times (-1)^{n} n! = 
     \binom{2n}{n} \times \frac{n!}{4^{n}}. 
\end{align*} 
The first identity in the proposition together with the previous 
half-index identity for the binomial coefficients imply the second 
congruence stated in the proposition. 
\end{proof} 

\subsubsection{Semi-polynomial congruences expanding the 
               central binomial coefficients} 

The next polynomial congruences satisfied by the central binomial coefficients 
modulo integer multiples of the individual 
polynomial powers, $n^p$, of $n$ for some fixed $p \geq 1$ 
also provide additional examples of some of the 
double-factorial-related phrasings of the expansions of 
\eqref{eqn_pnAlphaR_seqs_finite_sum_reps_modulop-stmts_v1} and 
\eqref{eqn_pnAlphaR_seqs_finite_sum_reps_modulop-stmts_v2} 
following from the noted identity given in 
\eqref{eqn_Vandermonde-like_PHSymb_exps_of_PhzCfs-stmt_v1} 
(see Lemma \ref{lemma_lm} and the 
computations contained in the reference \cite{SUMMARYNBREF-STUB})\footnote{ 
     \label{footnote_PolyMod_ArrowNotation} 
     In this context, the notation $\mod{f(x)} \looparrowright x \defmapsto n$ 
     denotes that the underlined expression should first be reduced as a 
     polynomial in $x$ modulo $f(x)$, and then after this operation is 
     performed that $x$ should be set to the explicit value of $n$. 
     Surprisingly, this procedure inplemented in \Mm{} using the 
     function \texttt{PolynomialMod} in the reference \cite{SUMMARYNBREF-STUB} 
     produces correct, integer-valued congruence expressions even when the 
     double-factorial-related numerator of the first equation in 
     \eqref{eqn_Wolstenholme-like_congruences_for_central_binomials} is 
     divided through by the reciprocal of $n!$, though we would typically 
     expect this not to necessarily be the case if 
     $2^n (2n-1)!!$ were first reduced modulo the right-hand-side 
     functions of $n$. 
}. 
\begin{align*} 
\tagonce\label{eqn_Wolstenholme-like_congruences_for_central_binomials} 
\binom{2n}{n} & = 
     \frac{2^{n}}{n!} \times (2n-1)!! && \\ 
     & \equiv 
     \left\lbrace 
     \undersetline{\mod{x^p} \quad \looparrowright \quad x \defmapsto n}{
     \sum_{i=0}^{n} \binom{x^p}{i} 
     2^{i} \Pochhammer{\mathsmaller{\frac{1}{2}} - x^p}{i} 
     \Pochhammer{\mathsmaller{\frac{1}{2}}}{n-i} 
     \times \frac{2^{2n}}{n!} 
     } 
     \right\rbrace 
     && \pmod{n^p} \\ 
\notag 
     & \equiv 
     \left\lbrace 
     \undersetline{\mod{x^p} \quad \looparrowright \quad x \defmapsto n}{
     \sum_{i=0}^{n} \binom{x^p}{i} \binom{2x-2i}{x-i} 
     \Pochhammer{\mathsmaller{\frac{1}{2}} - x^p}{i} \times 
     \frac{8^{i} \cdot (n-i)!}{n!} 
     } 
     \right\rbrace 
     && \pmod{n^p} 
\end{align*} 
The special cases of the congruences in 
\eqref{eqn_Wolstenholme-like_congruences_for_central_binomials} 
corresponding to $p \defequals 3$ and $p \defequals 4$, respectively, 
are related to the necessary condition for the primality of 
odd integers $n > 3$ in Wolstenholme's theorem and to the sequence of 
\emph{Wolstenholme primes} defined as 
\cite[\S 2.2]{PRIMEREC} \cite[\cf \S 7]{HARDYWRIGHTNUMT} 
(\seqnum{A088164}) 
\begin{align*} 
\tagtext{Wolstenholme Primes} 
\WolstPrimeSet & \defequals \left\{ n \geq 5: 
     \text{ $n$ prime \ and \ } 
     \mathsmaller{\binom{2n}{n} \equiv 2 \pmod{n^4}} 
     \right\} \\ 
     & \phantom{:} = 
     \left(16843, 2124679, \ldots \right). 
\end{align*} 

\subsubsection{An identity for the single factorial function 
               involving expansions of double factorial functions} 

As another example of the applications of these new integer congruence 
applications expanded through the double factorial function, 
notice that the 
following identity gives the form of another exact, finite double sum 
expansion of the single factorial function over 
convolved products of the double factorials: 
\begin{align*} 
(n-1)! & = (2n-3)!! + 
     \sum_{k=1}^{n-2} \sum_{j=k}^{n-1} 
     (-1)^{j+1} \Pochhammer{-j}{k} \Pochhammer{-(2n-k-j-2)}{j-k} 
     (2n-2j-3)!! \\ 
     & \phantom{ = (2n-3)!!} + 
     \sum_{k=1}^{n-2} \sum_{j=k+1}^{n-1} 
     (-1)^{j} \Pochhammer{-j}{k+1} \Pochhammer{-(2n-k-j-3)}{j-k-1} 
     (2n-2j-3)!! 
\end{align*} 
This identity is straightforward to prove 
starting from the first non-round sum given in 
{\S 5.1} of the reference combined with second identity for the 
component summation terms 
in {\S 6.3} of the same article \cite{DBLFACTFN-COMBIDENTS-SURVEY}. 
A modified approach involving the congruence techniques 
outlined in either the first cases cited in 
Example \ref{example_TripleSumIdents_App_to_WThm}, 
or as suggested in the previous few example cases from the 
last subsections of the article, 
then suggests even further applications adapting the 
results for new variants of the 
established, or otherwise well-known, special case 
congruence-based identities 
expanded for the notable prime number subsequences cited above in terms of the 
double factorial function 
(see Example \ref{subsubSection_example_PrimeSubsequences_ImmediateAppsOfWThm} 
in the next subsection). 

\subsection{Applications of Wilson's theorem in 
            other famous and notable special case prime subsequences} 
\label{subsubSection-Examples_SomeResults_for_Prime_k-Tuples} 
\label{subsubSection_remark_OtherApps_of_WThm_and_NewCongProps_to_PrimeSubseqs} 

The integer congruences obtained from Wilson's theorem for the 
particular special sequence cases noted in 
Section \ref{subsubSection_Examples_ConsequencesOfWThm} 
are easily generalized to give constructions over the forms other 
prime subsequences including the following special cases: 
\begin{enumerate} 
     \renewcommand{\itemsep}{-1mm} 

\item 
The \quotetext{\emph{Pierpont primes}} 
of the form $p \defequals 2^{u} 3^{v} + 1$ 
for some $u, v \in \mathbb{N}$ 
(\seqnum{A005109}); 

\item 
The subsequences of primes of the form $p \defequals n 2^{n} \pm 1$ 
(\seqnum{A002234}, \seqnum{A080075}); 

\item 
The \quotetext{\emph{Wagstaff primes}} 
corresponding to prime pairs of the form 
$\left(p, \frac{1}{3}(2^{p} + 1)\right) \in \mathbb{P}^{2}$ 
(\seqnum{A000978}, \seqnum{A123176}); and 

\item 
The generalized cases of the multifactorial prime sequences tabulated 
as in the reference \cite[Table 6, \S 2.2]{PRIMEREC} 
consisting of prime elements of the form 
$p \defequals \MultiFactorial{n}{\alpha} \pm 1$ for a fixed 
integer-valued $\alpha \geq 2$ and some $n \geq 1$. 

\end{enumerate} 

\begin{example}[The Factorial Primes and the 
                Fermat Prime Subsequences]
\label{subsubSection_example_PrimeSubsequences_ImmediateAppsOfWThm} 
The sequences of \emph{factorial primes} 
of the form $p \defequals n! \pm 1$ for some $n \geq 1$ 
satisfy congruences of the 
following form modulo $n! \pm 1$ given by the expansions of 
\eqref{eqn_Vandermonde-like_PHSymb_exps_of_PhzCfs-stmt_v1} and 
\eqref{eqn_Chn_formula_stmts} 
(\cite[\cf \S 2.2]{PRIMEREC}, \seqnum{A002981}, \seqnum{A002982}): 
\begin{align*} 
\tagtext{Factorial Prime Congruences} 
n! + 1 \text{ prime } 
     & \iff && \\ 
     & \phantom{\iff} 
     \underset{\mathlarger{(n!)! \equiv C_{n!+1,n!}(1, 1) \pmod{n!+1}}}{ 
     \underbrace{ 
     \sum_{i=0}^{n!} \binom{n!+1}{i}^{2} (-1)^{i} i! (n!-i)!} 
     } 
     & \equiv -1 && \pmod{n!+1} \\ 
n! - 1 \text{ prime } 
     & \iff 
     \underset{\mathlarger{(n!-2)! \equiv C_{n!-1,n!-2}(1, 1) \pmod{n!-1}}}{ 
     \underbrace{ 
     \sum_{i=0}^{n!-2} \binom{n!-1}{i}^{2} (-1)^{i} i! (n!-2-i)!} 
     } 
     & \equiv -1 && \pmod{n!-1}. 
\end{align*} 
The \emph{Fermat numbers}, $F_n$, 
generating the subsequence of \emph{Fermat primes} of the form 
$p \defequals 2^{m}+1$ where $m = 2^{n}$ for some $n \geq 0$ 
similarly satisfy the next congruences expanded through the 
identities for the single and double factorial functions given above 
\cite[\S 2.6]{PRIMEREC} \cite[\S 2.5]{HARDYWRIGHTNUMT} 
(\seqnum{A000215}, \seqnum{A019434}): 
\begin{align*} 
\tagtext{Fermat Prime Congruences} 
F_n \defequals 2^{2^n}+1 \text{ prime } 
     & \iff 
     2^{2^n}+1 \mid \left(2^{2^n}\right)! + 1 & && \\ 
     & \iff 
     2^{2^n}+1 \mid 2^{2^{2^n-1}} 
     \sum_{i=0}^{2^{2^n}} \binom{2^{2^n}+1}{i}^2 (-1)^{i} 
     i! \left(2^{2^n}-i\right)! +1 & && \\ 
     & \iff 
     2^{2^n}+1 \mid 2^{2^{2^n-1}} 
     \left(2^{2^{n}-1}\right)! \left(2^{2^n}-1\right)!! + 1 & && \\ 
     & \iff 
     2^{2^n}+1 \mid 2^{\frac{3}{4} \cdot 2^{2^n}} 
     \left(2^{2^{n}-2}\right)! \left(2^{2^n-1}-1\right)!! 
     \left(2^{2^n}-1\right)!! + 1 & && \\ 
     & \iff 
     2^{2^n}+1 \mid 2^{\frac{7}{8} \cdot 2^{2^n}} 
     \left(2^{2^{n}-3}\right)! 
     \left(2^{2^n-2}-1\right)!! 
     \left(2^{2^n-1}-1\right)!! 
     \left(2^{2^n}-1\right)!! + 1. & && 
\end{align*} 
For integers $h \geq 2$ and $r \geq 1$ such that $2^{r} \mid h$, the 
expansions of the congruences in the previous several equations correspond to 
forming the products of the single and double factorial functions 
modulo $h+1$ from the previous examples to require that 
\begin{align*} 
2^{\left(1-2^{-r}\right) \cdot h} \times \left(\frac{h}{2^{r}}\right)! 
     \left(\frac{h}{2^{r-1}} -1\right)!! 
     \times \cdots \times 
     \left(\frac{h}{2} -1\right)!! 
     \left(h -1\right)!! & \equiv -1 \pmod{h+1}, 
\end{align*} 
though more general expansions by products of the $\alpha$-factorial 
functions, $\MultiFactorial{\alpha}{n}$, for $\alpha > 2$ are apparent 
\cite[\cf \S 6.4]{MULTIFACT-CFRACS}. 
The \emph{generalized Fermat numbers}, 
$F_n(\alpha) \defequals \alpha^{2^n}+1$, and the corresponding 
\emph{generalized Fermat prime} subsequences 
when $\alpha \defequals 2, 4, 6$ suggest generalizations of the 
approach to the results in the previous equations through the 
procedure to the multiple, $\alpha$-factorial function expansions suggested in 
Section \ref{subsubSection_MoreGeneralExps_congruences_multiple_factfns} 
that generalizes the procedure to expanding the congruences above for the 
Fermat primes when $\alpha \defequals 2$. 
\end{example} 

The next concluding examples provide an approach to generalized congruences 
providing necessary and sufficient conditions on the primality of 
integers in special prime subsequences involving mixed expansions of the 
single and double factorial functions. 

\begin{example}[Mersenne Primes] 
\label{example_PrimeSubsequences_ImmediateAppsOfWThm_v2} 
The \emph{Mersenne primes} correspond to prime pairs of the 
form $(p, M_p)$ for $p$ prime and where 
$M_n \defequals 2^{n}-1$ is a \emph{Mersenne number} for some 
(prime) integer $n \geq 2$ 
(\cite[\S 2.7]{PRIMEREC}, \cite[\S 2.5; \S 6.15]{HARDYWRIGHTNUMT}, 
\cite[\cf \S 4.3, \S 4.8]{GKP}, 
\seqnum{A001348}, \seqnum{A000668}, \seqnum{A000043}). 
The requirements in Wilson's theorem for the primality of both 
$p$ and $M_p$ provide elementary proofs of the following equivalent 
necessary and sufficient conditions for the primality of the 
prime pairs of these forms: 
\begin{align*} 
\tagtext{Mersenne Prime Congruences} 
\left(p, 2^{p}-1\right) \in \mathbb{P}^{2} 
     & \iff && \\ 
     & \phantom{\iff\ } 
     p(2^p-1) \mid (p-1)! (2^p-2)! + (p-1)! + (2^p-2)! + 1 && \\ 
     & \iff 
     p(2^p-1) \mid \bigl( 
     C_{p(2^p-1),p-1}(-1, p-1) C_{p(2^p-1),2^p-2}(-1, 2^p-2) && \\ 
     & \phantom{\iff p(2^p-1) \mid \bigl( \quad } + 
     C_{p(2^p-1),p-1}(-1, p-1) + C_{p(2^p-1),2^p-2}(-1, 2^p-2) + 1 
     \bigr) && \\ 
     & \iff 
     p(2^p-1) \mid \bigl( 
     C_{p(2^p-1),p-1}(1, 1) C_{p(2^p-1),2^p-2}(1, 1) && \\ 
     & \phantom{\iff p(2^p-1) \mid \bigl( \quad } + 
     C_{p(2^p-1),p-1}(1, 1) + C_{p(2^p-1),2^p-2}(1, 1) + 1 
     \bigr) && \\ 
     & \iff 
     p(2^p-1) \mid 2^{2^{p-1}-1} (p-1)! (2^{p-1}-1)! (2^{p}-3)!! + 
     (p-1)! + (2^p-2)! + 1. && 
\end{align*} 
The congruences on the right-hand-sides of the previous equations 
are then expanded by the results in 
\eqref{eqn_Vandermonde-like_PHSymb_exps_of_PhzCfs-stmt_v1} and 
\eqref{eqn_Chn_formula_stmts}, and 
through the second cases of the more general 
product function congruences stated in 
\eqref{eqn_pnAlphaR_seqs_finite_sum_reps_modulop-stmts_v1}. 
\end{example} 

\begin{example}[Sophie Germain Primes] 
Wilson's theorem similarly implies the next related 
congruence-based characterizations of the 
\emph{Sophie Germain primes} corresponding to the 
prime pairs of the form $(p, 2p+1)$ where $p, p-1, 2p < p(2p+1)$ 
(\cite[\S 5.2]{PRIMEREC}, \seqnum{A005384}). 
\begin{align*} 
\tagtext{Sophie Germain Prime Congruences} 
\left(p, 2p+1\right) \in \mathbb{P}^{2} & & \\      
     & \iff 
     (p-1)! (2p)! + (p-1)! + (2p)! & \equiv -1 & \pmod{p(2p+1)} \\ 
     & \iff 
     2^p p! (p-1)! (2p-1)!! + (p-1)! + (2p)! & \equiv -1 
     & \pmod{p(2p+1)} 
\end{align*} 
The expansions of the generalized forms of the Sophie Germain primes 
noted in the reference \cite[\S 5.2]{PRIMEREC} 
also provide applications of the 
multiple, $\alpha$-factorial function identities outlined in 
Example \ref{subsubSection_example_PrimeSubsequences_ImmediateAppsOfWThm} and 
suggested in 
Section \ref{subsubSection_MoreGeneralExps_congruences_multiple_factfns} 
of the article 
which result from expansions of the arithmetic progressions 
of the single factorial functions given in the examples from the 
reference \cite[\S 6.4]{MULTIFACT-CFRACS}. 
\end{example} 

\begin{remark}[Congruences for Integer Powers and Sequences of Binomials] 
Notice that most of the factorial function expansions involved in the 
results formulated by the previous few examples do not 
immediately imply corresponding congruences obtained from 
\eqref{eqn_pnAlphaR_seqs_finite_sum_reps_modulop} 
satisfied by the \emph{Wieferich prime} sequence defined by 
\cite[\S 5.3]{PRIMEREC} (\seqnum{A001220}) 
\begin{align*} 
\tagtext{Wieferich Primes} 
\WieferichPrimeSet & \defequals 
     \left\{ n \geq 2: \text{ $n$ prime \ and \ } 
     2^{n-1} \equiv 1 \pmod{n^2} 
     \right\} \\ 
     & \phantom{\defequals} \quad 
     \seqmapsto{A001220} \left(1093, 3511, \ldots \right), 
\end{align*} 
nor results for the variations of the sequences of 
binomials enumerated by the rational convergent-function-based 
generating function identities over the 
binomial coefficient sums constructed in the reference 
\cite[\S 6.7]{MULTIFACT-CFRACS} 
modulo prime powers $p^m$ for $m \geq 2$. 
However, indirect expansions of the sequences of binomials, 
$2^{n-1}$ and $2^{n-1}-1$, by the 
Stirling numbers of the second kind through the lemma provided in the reference 
\cite[Lemma 12]{MULTIFACT-CFRACS}, yield the following 
divisibility requirements characterizing the sequence of 
Wieferich primes where $\Pochhammer{2}{n} = (n+1)!$: 
\begin{align*} 
2^{n-1} \phantom{-1} 
     & = 
     \sum_{k=0}^{n-1} \gkpSII{n-1}{k} (-1)^{n-1-k} \Pochhammer{2}{k} && \\ 
     & \equiv 
     \sum_{k=0}^{n-1} \sum_{i=0}^{k+1} 
     \gkpSII{n-1}{k} \binom{n^2}{i}^{2} (-1)^{n-1-k-i} i! (k+1-i)! && 
     \pmod{n^2} \\ 
2^{n-1}-1 
     & = 
     2^{n-2} + 2^{n-3} + \cdots + 2 + 1 && \\ 
     & = 
     \sum_{\substack{ 0 \leq j \leq i \leq n-2}} 
     \gkpSII{i}{j} (-1)^{i-j} \Pochhammer{2}{j} && \\ 
     & \equiv 
     \sum_{\substack{ 0 \leq m \leq j \leq i \leq n-2}} 
     \gkpSII{i}{j} \binom{n^2}{m} (-1)^{i-j+m} 
     \FFactII{(n^2+1)}{m} \Pochhammer{2}{j-m} && 
     \pmod{n^2}. 
\end{align*} 
The constructions of the corresponding congruences 
for the sequences of binomials, $a^{n-1} - 1 \pmod{n^2}$, are 
obtained by a similar procedure \cite[\cf \S 5.3; Table 45]{PRIMEREC}. 
Additional expansions of related congruences for the terms 
$3^{t} + 1 \pmod{2t+1}$ for prime $2t+1 \in \mathbb{P}$ 
in the particular forms of 
\begin{align*} 
3^{t} + 1 
     & \equiv 
     \phantom{4 \times} 
     \sum_{j=0}^{t} \sum_{i=0}^{j} 
     \gkpSII{t}{j} \binom{2t+1}{m} (-1)^{t-j+m} 
     \FFactII{(2t+3)}{m} \Pochhammer{3}{j-m} + 1 && 
     \pmod{2t+1} \\ 
3^{t} + 1     
     & \equiv 
     4 \times \sum_{m=0}^{t-1} \sum_{j=0}^{m} \sum_{i=0}^{j} 
     \gkpSII{m}{j} \binom{2t+1}{i} (-1)^{t-1-j+i} 
     \FFactII{(2t+3)}{i} \Pochhammer{3}{j-i} && 
     \pmod{2t+1}, 
\end{align*} 
where $\Pochhammer{3}{j} = \frac{1}{2} (i+2)!$, 
lead to double and triple sums providing the necessary and sufficient 
condition for the primality of the Fermat primes, $F_k$, 
from the reference \cite[\S 6.14]{HARDYWRIGHTNUMT} 
when $t \defmapsto 2^{2^{k}-1}$ for some integer $k \geq 1$. 
\end{remark} 

\section{Conclusions} 
\label{Section_ConcludingRemarks} 

\subsection{Summary} 

In Section \ref{Section_KeyProp_Proof}, we proved the key results stated in 
Proposition \ref{prop_KeyProp} for the 
special cases where $\alpha := \pm 1, 2$, which includes the congruences 
involving the single and double factorial functions cited in the other 
applications from Section \ref{subSection_FiniteDiffEqns_for_the_GenFactFns}. 
We note that our numerical evidence provided in the summary notebook 
reference \cite{SUMMARYNBREF-STUB} suggests that these congruences do 
in fact also hold more generally for all integers $\alpha \neq 0$ and 
all $h \geq 2$ (not just the odd and prime cases of these integer moduli). 
The applications of the key proposition given in the article provide 
a number of propositions which follow as corollaries of the first set of 
results. The specific examples of the new results we prove within the 
article include new finite sum expansions and congruences for the 
$\alpha$-factorial functions, as well as applications of our results to 
formulating new statements of necessary and sufficient conditions on the 
primality of integer subsequences, pairs, and triples. 

In many respects, this article is a follow-up to the first article 
published in $2017$ \cite{MULTIFACT-CFRACS}. 
For comparision with the results given in the reference, we note that the 
results in Proposition \ref{prop_KeyProp} and 
Corollary \ref{cor_CongForThe_AlphaFactFns} 
do not provide a simple or otherwise apparent mechanism for formulating 
new congruences and recurrence relations satisfied by the 
triangles of $\alpha$-factorial coefficients, $\FcfII{\alpha}{n}{k}$. 
However, such congruences and recurrence relations for the 
corresponding coefficients, $[R^k] p_n(\alpha, R)$, of the 
generalized product sequences defined in 
\eqref{eqn_GenFact_product_form} 
are in contrast easy to obtain from the results in the 
key proposition. We conclude the article by posing 
several remaining open questions related to the 
expansions of the generalized factorial functions considered in 
this article and in the references \cite{MULTIFACT-CFRACS,MULTIFACTJIS}. 

\subsection{Open questions and topics for future research} 

\subsubsection{Open questions} 

We pose the following open questions as topics for future research on the 
generalized factorial functions studied in this article and in the 
references \cite{MULTIFACT-CFRACS,MULTIFACTJIS}: 
\begin{enumerate} 

\item Can we determine bounds on the zeros of the convergent denominator 
      functions, $\FQ_h(\alpha, \alpha-d; z)$ and 
      $\FQ_h(-\alpha, \alpha n-d; z)$, in order to determine more accurate 
      Stirling-like approximations for the generalized multiple factorial 
      function cases of $\AlphaFactorial{\alpha n-d}{\alpha}$? 

\item What is the best way to evaluate the $\alpha$-factorial functions, 
      $\MultiFactorial{n}{\alpha}$, for fractional $\alpha$ or non-negative 
      rational $n$? For example, are the finite sum representations in 
      Proposition \ref{prop_KeyProp} a good start at 
      generalizing these functions to non-integer arguments and the 
      strictly rational-valued parameters that occur in applications 
      intentionally not discussed in these articles? 

\item What is the best way to translate the new congruence results for the 
      single, double, and $\alpha$-factorial functions to form integer 
      congruences for the binomial coefficients, 
      $\binom{x}{k} = \frac{\FFactII{x}{k}}{k!}$? 

\end{enumerate} 

\subsubsection{Applications to generalizations of known finite sum 
               identities involving the double factorial function} 
\label{subsubSection_FutureResTopics_GenDblFactFnSumIdents_FiniteSums} 

The construction of further 
analogues for generalized variants of the finite summations and 
more well-known combinatorial identities satisfied by the 
double factorial function cases when $\alpha \defequals 2$ from the 
references is suggested as a topic for future investigation. 
The identities for the more general $\alpha$-factorial function cases 
stated in 
Example \ref{example_GenDblFactFnSumIdents_FiniteSumsInvolving_AlphaFactFns} of 
Section \ref{ssS_example_GenDblFactFnSumIdents_FiniteSumsInvolving_AlphaFactFns} 
suggest one possible approach to generalizing the 
known identities summarized in the references 
\cite{MAA-FUN-WITH-DBLFACT,DBLFACTFN-COMBIDENTS-SURVEY} for the 
next few particularly interesting special cases corresponding to the 
triple and quadruple factorial function cases, $n!!!$ and $n!!!!$, 
respectively.

\bigskip
\hrule
\bigskip

\noindent \textit{2010 MSC}: 
Primary 05A10; Secondary 11Y55, 11Y65, 11A07, 11B50.  

\smallskip 
\noindent\textit{Keywords}: 
continued fraction, J-fraction, 
Pochhammer symbol, factorial function, 
multifactorial, multiple factorial, 
single factorial, double factorial, triple factorial, 
Pochhammer k-symbol, 
factorial congruence, prime congruence, Wilson's theorem, 
Clement's theorem, sexy prime. 

\bigskip
\hrule
\bigskip 

\noindent 
(Concerned with sequences
\seqnum{A000043}, \seqnum{A000215}, \seqnum{A000668}, \seqnum{A000978}, 
\seqnum{A000984}, \seqnum{A001008}, \seqnum{A001097}, \seqnum{A001220}, 
\seqnum{A001348}, \seqnum{A001359}, \seqnum{A002144}, \seqnum{A002234}, 
\seqnum{A002496}, \seqnum{A002805}, \seqnum{A002981}, \seqnum{A002982}, 
\seqnum{A005109}, \seqnum{A005384}, \seqnum{A007406}, \seqnum{A007407}, 
\seqnum{A007408}, \seqnum{A007409}, \seqnum{A007540}, \seqnum{A007619}, 
\seqnum{A019434}, \seqnum{A022004}, \seqnum{A022005}, \seqnum{A023200}, 
\seqnum{A023201}, \seqnum{A023202}, \seqnum{A023203}, \seqnum{A046118}, 
\seqnum{A046124}, \seqnum{A046133}, \seqnum{A080075}, \seqnum{A088164}, 
\seqnum{A123176}. 
) 

\bigskip
\hrule
\bigskip 

\end{document}